\documentclass[sn-mathphys,Numbered]{sn-jnl}

\usepackage{tikz,pgfplots}
\pgfkeys{/pgf/declare function={arccosh(\x) = ln(\x + sqrt(\x^2-1));}}

\usepackage{graphicx}%
\usepackage{multirow}%
\usepackage{amsmath,amssymb,amsfonts}%
\usepackage{amsthm}%
\usepackage{mathrsfs}%
\usepackage[title]{appendix}%
\usepackage{xcolor}%
\usepackage{textcomp}%
\usepackage{manyfoot}%
\usepackage{booktabs}%
\usepackage{algorithm}%
\usepackage{algorithmicx}%
\usepackage{algpseudocode}%
\usepackage{listings}%
\usepackage{verbatim}
\usepackage{subfigure}
\usepackage{parskip}

\theoremstyle{thmstyleone}%

\theoremstyle{thmstyletwo}%

\theoremstyle{thmstylethree}%
\newtheorem{definition}{Definition}%

\newtheorem{theorem}{Theorem}[section]
 
\newtheorem{lemma}[theorem]{Lemma}

\begin{document}

\title[Synchronization in the quaternionic Kuramoto model]{Synchronization in the quaternionic Kuramoto model}

\author*[1]{\fnm{Ting-Yang} \sur{Hsiao}}\email{tyhsiao2@illinois.edu}

\author[2]{\fnm{Yun-Feng} \sur{Lo}}\email{ylo49@gatech.edu}

\author[3]{\fnm{Winnie} \sur{Wang}}\email{wwang629@wisc.edu}

\affil*[1]{\orgdiv{Department of Mathematics}, \orgname{University of Illinois Urbana-Champaign}, \orgaddress{\street{}\city{Champaign}, \postcode{61820}, \state{IL}, \country{USA}}}

\affil[2]{\orgdiv{School of Electrical and Computer Engineering}, \orgname{Georgia Institute of Technology}, \orgaddress{\street{}\city{Atlanta}, \postcode{30332}, \state{GA}, \country{USA}}}

\affil[3]{\orgdiv{Department of Physics}, \orgname{University of Wisconsin-Madison},\orgaddress{\street{} \city{Madison}, \postcode{53706}, \state{WI}, \country{USA}}}


\abstract{In this paper, we propose an $N$ oscillators Kuramoto model with quaternions $\mathbb{H}$. In case the coupling strength is strong, a sufficient condition of synchronization is established for general $N\geqslant 2$. On the other hand, we analyze the case when the coupling strength is weak. For $N=2$, when coupling strength is weak (below the critical coupling strength $\lambda_c$), we show that new periodic orbits emerge near each equilibrium point, and hence phase-locking state exists. This phenomenon is different from the real Kuramoto system since it is impossible to arrive at any synchronization when $\lambda<\lambda_c$. We prove a theorem that states a set of closed and dense contour forms near each equilibrium point, resembling a tree's growth rings. In other words, the trajectory of phase difference lies on a $4D$-torus surface. Therefore, this implies that the phase-locking state is Lyapunov stable but not asymptotically stable.  The proof uses a new infinite buffer method (``$\delta/n$ criterion") and a Lyapunov function argument. This has been studied both analytically and numerically. For $N=3$, we consider the  ``Lion Dance flow", the analog of Cherry flow for our model, to demonstrate that the quaternionic synchronization exists even when the coupling strength is ``super weak" (when $\lambda/\omega <0.85218915...$). Also, numerical evaluation reveals that when $N>3$, the stable manifold of Lion Dance flow exists, and the number of these equilibria is $\lfloor \frac{N-1}{2}\rfloor$. Therefore, we conjecture that Lyapunov stable quaternionic synchronization always exists.}






\maketitle

\section{Introduction}\label{sec1}
Synchronization is the natural phenomenon of collective oscillation observed in all autonomous oscillators in a system, regardless of the initial phases and frequencies of each unit. The study of synchronization has a long history, starting with Horologium Oscillatorium written by Huygens in 1673 \cite{horologium}. This problem has been widely studied since then. Please see general description articles \cite{reviewstrogatz2001, bookstrogatz2004}, textbooks \cite{bookpikovsky2001, bookosipov2007}, and review articles \cite{reviewstrogatz2000, reviewboccaletti2002}.  

The first attempt at understanding this problem came from Winfree \cite{reviewwinfree1966}, who studied the nonlinear dynamics of a large population of weakly coupled limit cycle oscillators. He used a mean-field approximation approach and it demonstrated that if each oscillator was coupled to the system's collective rhythm, then a ``phase transition" of frequencies would occur. However, Winfree's model was not popular because it was intractable, especially in the limit of large populations \cite{reviewstrogatz2000}. Kuramoto gradually refined Winfree's model: in 1975, he presented his first approach to the synchronization problem. Kuramoto's model \cite{kuramoto1975self} described a system of weakly coupled and nearly identical interacting limit cycle oscillators where each influences the phase of connected oscillators. His model was proven more popular than Winfree's due to its ease in displaying various synchronization behaviors. For a system of N oscillators, the Kuramoto equations are usually written as
\begin{align}\label{first}
{\dot{\theta}}_n = \omega_n + \frac{\lambda}{N} \sum^{N}_{m=1} \text{sin}(\theta_m - \theta_n)
\end{align}
where $\theta_n$ is the real phase angle of the \textit{n}th oscillator and $\omega_n$ is its natural frequency, and  $\lambda$ is the coupling strength. This model shows that synchronization happens when the coupling strength surpasses a certain threshold.

Past work on expanding algebras on existing problems is a great motivator for new problems. For instance, it was demonstrated that the extension from real numbers to more general algebraic structures was fundamentally significant both in sciences and engineering. When classical Newtonian mechanics was notably superseded by Einstein's special relativity in Einstein's 1905 paper \cite{Einstein1905}, which demonstrated that physics in a four-dimensional pseudo-Riemannian manifold is more fundamental. 

Efforts to expand the Kuramoto model to complex numbers were done as far back as 1984 when it was expanded to include a complex order parameter to describe ``spontaneous" synchronization \cite{bookkuramoto1984}. Rigorous analysis of this order parameter was demonstrated in Strogatz's work \cite{reviewstrogatz2001, reviewstrogatz2000}. Various works have applied this order parameter to demonstrate synchronization \cite{PhysRevE.76.057201, PhysRevE.72.046211, 1383983, Kincomplex, roberto2022, roberto2023, ha4}. In 2011, D{\"o}rfler and Bullo summarized different levels of synchronization and proved the necessary and sufficient conditions for classical Kuramoto problems and analyzed phase transitions of the system \cite{dorfler2011critical, dorfler2012synchronization, dorfler2014synchronization}. Several authors have also suggested the feasibility of a complexified Kuramoto model \cite{thumler2023synchrony} or a complex network \cite{bottcher2023complex}.

Moreover, there are currently many diverse efforts to further our understanding synchronization through the Kuramoto model, such as demonstrating the viability of expanding the Kuramoto model into non-Abelian generalizations \cite{lohe, deville2019, ha, ha2, ha3}, applying seminorms to describe frequency synchronization \cite{bronski2021} and analytically expressing frequency synchronization solutions using distributions \cite{bronski2011}.

For our work, we consider three characteristics of this problem: firstly ``phase locking", the boundedness of maximum phase difference; secondly, ``phase synchronization", the convergence of phase differences to zero; and lastly, ``frequency synchronization" when all frequencies converge to a constant. Then, we expand the Kuramoto model with quaternions $\mathbb{H}$, which are numbers of a noncommutative division algebra represented in the form $$w+x\mathbf{i}+y\mathbf{j}+z\mathbf{k},$$ where $w,x,y,z\in\mathbb{R}$ and $\mathbf{i},\mathbf{j},\mathbf{k}$ are basis satisfying $\mathbf{i}^2=\mathbf{j}^2=\mathbf{k}^2=\mathbf{ijk}=-1$. Although quaternion algebra is a choice and not necessary for researching problems in natural phenomena, several authors have begun to investigate theories related to quaternion algebra, such as special relativity, electrodynamics \cite{de1996quaternions, weiss1940some, saue2005spin, ward2012quaternions}, and Pauli spinors \cite{cahay2019quaternion, ilamed1981algebras, saue1999quaternion}. In 1944, Onsager's paper \cite{Onsager1944} introduced quaternion algebra to analyze the phase transition in the Ising model. Many theoretical analyses on the condensed matter research also appeal to quaternions \cite{li2013high, PhysRevLett.123.254101, PhysRevLett.125.053601, PhysRevLett.125.033901, PhysRevLett.128.246601, PhysRevLett.127.170602}. Other applications of quaternionic structure are in biology, especially in DNA digital data storage \cite{chang2012dna, AndrzejK, 1416318, 4365821, Iqbal} and applications in computer graphics  \cite{pletinckx1989quaternion, vince2011quaternions}. 

Here, we distinguish the solutions for the strong and weak cases, where $\lambda > \lambda_c$ and $\lambda < \lambda_c$ for strong and weak respectively (Definition \ref{def 5}). In the strong case, the Lyapunov energy method for frequency synchronization is used. In 1993, Van Hemmen and Wreszinski \cite{van1993lyapunov} introduced Lyapunov function to demonstrate that phase-locking and frequency synchronization are equivalent in fully connected topology. Recent work has demonstrated the tractability of this method: Hsia and his collaborators used this Lyapunov energy method for real Kuramoto models with inertia, external control, and different topological structures \cite{hsia2019synchronization, chen2021synchronization, chen2022frequency, chen2020mathematical}. 

In this paper, we propose a quaternionic Kuramoto model. For the strong coupling strength, we applied an appropriate Lyapunov energy function to control the real part of each oscillator which is shown in Section \ref{sec 3}. For the case when coupling strength is weak, we apply the new notions stated in Section \ref{sec 4}.

We organize the paper in the following order. In Section \ref{sec 2} we introduce the quaternionic Kuramoto model \eqref{Q K model}. In Section \ref{sec 2.1},  we use Pauli matrices as matrix representations of quaternions $\mathbb{H}$ to define fundamental functions of quaternions (Theorem \ref{thm for sine}). Then we obtain the main system of equations stated in \eqref{main eq}. In Section \ref{sec 2.2}, we clarify different levels of synchronization for the quaternionic Kuramoto model (Definition \ref{def 2}-\ref{def 4}). It is worth mentioning that phase-locking (Definition \ref{def 2}) and frequency synchronization (Definition \ref{def 3}) are equivalent in the real Kuramoto model but this equivalence does not hold in the quaternionic Kuramoto model. We also introduce a critical coupling strength, $\lambda_c$, in Definition \ref{def 5}. By using $\lambda_c$, we can separate the quaternionic Kuramoto model into strong (Section \ref{sec 3}) and weak cases (Section \ref{sec 4}).

In Section \ref{sec 3}, we consider the case when the coupling strength is strong for $N \geqslant 2$ oscillators. Since the real and imaginary oscillations can interact with each other, we first describe the oscillations of both parts individually by alternating the control between the real and imaginary parts. We restrict the difference of the real part between each oscillator. Then, we demonstrate that the oscillators arrive at frequency and phase synchronization exponentially in each of the three imaginary parts. Once we achieve this, we demonstrate that the system achieves frequency synchronization by using the Lyapunov energy function \eqref{H function} stated in Theorem \ref{F sync}. We also prove that the system achieves phase synchronization (Theorem \ref{3.5}) when all of the natural frequencies are the same.

In Section \ref{sec 4}, we consider the weak coupling case. The two-oscillators ($N=2$) case is analyzed in Section \ref{sec 4.1}. It is well-known that it is impossible to arrive at any synchronization when the coupling strength is weak, i.e., $\lambda<\lambda_c$. However, in the quaternionic Kuramoto model, periodic orbits emerge near each equilibrium point and the trajectories of the phase difference for these two oscillators lie on a $4$ dimensional ``peach ring" (Fig.~\ref{peachring}). We introduce the so-called $\delta/n$  criterion (Lemma \ref{delta n criterion}) and use a symmetry argument (Lemma \ref{sym prop}) to allow control of the trajectories and show that every trajectory with initial condition located above the equilibria forms a closed contour. Once we obtain these results, we construct a deceleration region to demonstrate that trajectories form a layered contour that resembles a tree's growth rings, like a set of ``infinitely nested closed contours, one within another." In Section \ref{sec 4.2}, we analyze the ``Lion Dance flow", the analog of the ``Cherry flow" regime \cite{maistrenko2004mechanism}. First, we consider the ``super weak" case, i.e., $\lambda<\Lambda_c$ in Section \ref{4.2.1}. When the coupling strength is super weak, the system does not have any synchronization in real Kuramoto models. Then in Section \ref{4.2.1}, we use the horizontal cutting method to locate two eqilibria (Lemma \ref{lemma 4.5}) and Hartman–Grobman theorem to demonstrate that the stable manifold of phase-locking and frequency synchronization exist. 

Finally, when the coupling strength satisfies $\Lambda_c \leqslant \lambda<\lambda_c$ (Section \ref{4.2.2}), we show there is a quaternionic phase-locking. There's another phase-locking that is also observed in real Kuramoto models, which is semistable when $\lambda$ is equal to critically weak $\Lambda_c$. This phase-lock will bifurcate into two phase-locking states: one is stable while the other is unstable when the coupling strength $\lambda$ exceeds critically weak $\Lambda_c$.

\section{The quaternionic Kuramoto model}\label{sec 2}
Here we propose the Kuramoto model with quaternions as follows.
\begin{align} \label{Q K model}
    \dot{q}_n(t)=\omega_n+\dfrac{\lambda}{N}\sum_{m=1}^N \sin(q_m(t)-q_n(t)),
\end{align}
where $n\in\{1,...,N\}$. Here the $n$-th oscillator is represented by a phase variable $q_n=w_n+x_n \mathbf{i}+y_n \mathbf{j}+z_n \mathbf{k}\in\mathbb{H}$, and $\omega_n\in\mathbb{R}$ is its natural frequency. We denote $\lambda>0$ as coupling strength, which is a key parameter in this system. Considering an invariant manifold $x_n\equiv y_n\equiv z_n \equiv 0$, this equation \eqref{Q K model} reduces to the classical Kuramoto model in \cite{kuramoto1975self, bookkuramoto1984}. 

\subsection{Preliminaries.} \label{sec 2.1}
Before delving deeper into the main equation \eqref{sec 2}, let's establish the precise definitions of fundamental functions—namely, the exponential, sine, and cosine functions—within the context of quaternions. It is well-known that quaternions can be represented as $2\times 2$ complex matrices (see, for instance, \cite{voight2021quaternion, hanson2005visualizing}). 

Let us consider Pauli matrices as follows: 
\begin{align}
     \mathrm{I}=\begin{pmatrix}
    1 & 0 & \\   
    0 & 1 & 
\end{pmatrix}
    ~~\sigma_x=\begin{pmatrix}
    0 & 1 & \\   
    1 & 0 & 
\end{pmatrix},~~\sigma_y=\begin{pmatrix}
    0 & -i & \\   
    i & 0 & 
\end{pmatrix}, ~~\sigma_z=\begin{pmatrix}
    1 & 0 & \\   
    0 & -1 &
\end{pmatrix},
\end{align}
whose multiplication rules satisfy
\begin{align*}
&\sigma_x^2=\sigma_y^2=\sigma_z^2=\mathrm{I},~~~~~\sigma_x\sigma_y=-\sigma_y\sigma_x=i\sigma_z,\\
&\sigma_y\sigma_z=-\sigma_z\sigma_y=i\sigma_x,~ \sigma_z\sigma_x=-\sigma_x\sigma_z=i\sigma_y.
\end{align*}
Therefore, the matrix representation of quaternions $\mathbb{H}$ can be regarded as elements of $M_2(\mathbb{C})$ in the embedding with 
\begin{align}
    \phi:\mathbb{H}\mapsto M_2(\mathbb{C}),
\end{align}
where
\begin{align}
\phi(1)=\mathrm{I},~ \phi(\mathbf{i})=-i\sigma_x, ~\phi(\mathbf{j})=-i\sigma_y, ~\phi(\mathbf{k})=-i\sigma_z.
\end{align}
A straightforward computation reveals
\begin{align}
    \phi(w+x\mathbf{i}+y\mathbf{j}+z\mathbf{k})= 
\begin{pmatrix}
 &u        &-v  \\   
 &\bar{v}  &\bar{u}  
\end{pmatrix},
\end{align}
where $u:=w-zi$ and $v:=y+xi$, and ~$\bar{}$ ~denotes complex conjugation.  Let
\begin{align}
    M:=\Bigg\{\begin{pmatrix}
    a & b & \\   
    c & d & 
\end{pmatrix}\in M_2(\mathbb{C}): a=\bar{d}, ~ b=-\bar{c}\Bigg\},
\end{align}
so it is obvious that $\phi$ is isomorphic from $\mathbb{H}$ to $M$. We may define the exponential, sine, and cosine functions on a quaternion as follows.
\begin{definition}[Exponential, Sine, Cosine functions on $\mathbb{H}$]
Given $q\in\mathbb{H}$,
\begin{align*}
    \exp(q):=&\phi^{-1}\exp(\phi(q)),\\
    \sin(q):=&\phi^{-1}\sin(\phi(q)),\\
    \cos(q):=&\phi^{-1}\cos(\phi(q)).
\end{align*}
\end{definition}

We pause to remark that $\phi(q)$ is a $2\times 2$ matrix with complex elements, and for all $A\in M$ the fundamental functions of a matrix are defined as follows
\begin{align*}
    \exp(A):=\sum_{n=0}^{\infty} \dfrac{1}{n!} A^{n},~~\sin(A):=\sum_{n=0}^{\infty} \dfrac{(-1)^n}{(2n+1)!} A^{2n+1},
\end{align*}
and
\begin{align*}
    \cos(A):=\sum_{n=0}^{\infty} \dfrac{(-1)^n}{(2n)!} A^{2n}.
\end{align*}
A straightforward calculation immediately reveals the following theorem.
\\
\begin{theorem} \label{thm for sine}
Given $ q=w+x \mathbf{i}+y \mathbf{j}+ z\mathbf{k}\in \mathbb{H}$, the exponential, sine, cosine functions of a quaternion are\\
\begin{align*}
    e^q=e^w\left( \cos(\sqrt{x^2+y^2+z^2})+\dfrac{x i+y j+z k}{\sqrt{x^2+y^2+z^2}}\sin(\sqrt{x^2+y^2+z^2})\right),
\end{align*}
\begin{align*}
\sin(q)=\sin(w)\cosh(\sqrt{x^2+y^2+z^2})+\cos(w)\sinh(\sqrt{x^2+y^2+z^2})\dfrac{x\mathbf{i}+y\mathbf{j}+z\mathbf{k}}{\sqrt{x^2+y^2+z^2}},
\end{align*}
and
\begin{align*}
\cos(q)=\cos(w)\cosh(\sqrt{x^2+y^2+z^2})+\sin(w)\sinh(\sqrt{x^2+y^2+z^2})\dfrac{x\mathbf{i}+y\mathbf{j}+z\mathbf{k}}{\sqrt{x^2+y^2+z^2}}.
\end{align*}
\end{theorem}
\begin{proof}
    The proof is fundamental, so we only demonstrate the sine function. Given $q=w+x \mathbf{i}+y \mathbf{j}+ z\mathbf{k}\in \mathbb{H}$,
    \begin{align}
        \phi(q)=\begin{pmatrix}
 w-z i & -y-x i & \\   
 y-x i & ~~~w+z i & 
\end{pmatrix}=SJS^{-1},
    \end{align}
where 
\begin{align}
    S=\begin{pmatrix}
 \dfrac{z+\sqrt{x^2+y^2+z^2}}{x+iy } & \dfrac{z-\sqrt{x^2+y^2+z^2}}{x+iy } & \\   
 1 & 1 & 
\end{pmatrix},
\end{align}
and
\begin{align}
    J=\begin{pmatrix}
 w-i\sqrt{x^2+y^2+z^2} & 0 & \\   
 0 & w+i\sqrt{x^2+y^2+z^2} & 
\end{pmatrix}.
\end{align}
Denote $|v|=\sqrt{det(\phi(q))-w^2}$. Then a straightforward calculation shows that: 
\begin{align*}
    &\sin(\phi(q))=S\sin(J)S^{-1} \\
    =&\begin{pmatrix}
  \sin(w)\cosh(|v|)-i\cos(w)\sinh(|v|)\dfrac{z}{|v|}& \cos(w)\sinh(|v|)\dfrac{-y-i x}{|v|} & \\   
 \cos(w)\sinh(|v|)\dfrac{y-i x}{|v|} & \sin(w)\cosh(|v|)+i\cos(w)\sinh(|v|)\dfrac{z}{|v|} & 
\end{pmatrix}
\end{align*}
Therefore, we obtain
\begin{align*}
\sin(q)=\sin(w)\cosh(\sqrt{x^2+y^2+z^2})+\cos(w)\sinh(\sqrt{x^2+y^2+z^2})\dfrac{x\mathbf{i}+y\mathbf{j}+z\mathbf{k}}{\sqrt{x^2+y^2+z^2}}.
\end{align*}
We may follow a similar argument to find $e^q$ and $\cos(q)$, so the proof is complete.
\end{proof}

By means of the Theorem \ref{thm for sine}, the real and imaginary parts of equation \eqref{Q K model} can be separated and written as\\
\begin{equation}\label{main eq}
\left\{ \begin{aligned}
&\dot{w}_n=\omega_n+\dfrac{\lambda}{N}\sum_{m=1}^N \sin(w_m-w_n)\cosh(v_{mn}),~\\ 
&\dot{x}_n=~~~~~~~~\dfrac{\lambda}{N}\sum_{m=1}^N \cos(w_m-w_n)\sinh(v_{mn})\dfrac{x_m-x_n}{v_{mn}},\\ 
&\dot{y}_n=~~~~~~~~\dfrac{\lambda}{N}\sum_{m=1}^N \cos(w_m-w_n)\sinh(v_{mn})\dfrac{y_m-y_n}{v_{mn}},\\
&\dot{z}_n=~~~~~~~~\dfrac{\lambda}{N}\sum_{m=1}^N \cos(w_m-w_n)\sinh(v_{mn})\dfrac{z_m-z_n}{v_{mn}},\\
\end{aligned} 
\right. 
\end{equation}
where for each $(m,n)\in \{1,...,N\}^2$
\begin{align} \label{def of v_mn}
    v_{mn}:=\sqrt{|x_m-x_n|^2+|y_m-y_n|^2+|z_m-z_n|^2} .
\end{align}

\color{black}
\subsection{Definitions of synchronization} \label{sec 2.2}
In order to construct a sufficient condition on 
synchronized state for system equations \eqref{main eq}, we will first introduce some notations and define different levels of synchronization. Let
\begin{align}
 q(t)=(w(t),x(t),y(t),z(t)):=(w(t),\mathbf{v}(t)),
\end{align}
by abuse of the notation, denote a function $q:\mathbb{R}^+\rightarrow\mathbb{R}^{4N}$, and 
\begin{align}
 \boldsymbol{\omega}=(\omega_1,\omega_2,...,\omega_N)    
\end{align}
be a constant in $\mathbb{R}^{N}$. We observe that 
\begin{align}
    \dfrac{d}{dt} (w_1+...+w_N)=\omega_1+...+\omega_N.
\end{align}
This shows that the overall frequency for the real parts is the mean (average) of the natural frequencies. It is often useful to introduce a rotating frame by using the changes of the variables
\begin{align}
    w_n \mapsto w_n-t(\omega_1+...+\omega_N)/N.
\end{align}
This observation may allow us to assume 
\begin{align} \label{zero mean omega}
\omega_1+...+\omega_n=0,
\end{align}
without loss of generality. 
\\
\begin{definition} [Full phase-locking synchronization] \label{def 2}
    The quaternionic Kuramoto model \eqref{main eq} arrives at full phase-locking if 
    \begin{align}
        \sup_{t>0} \max_{n,m\in\{1,...N\}}|q_n(t)-q_m(t)|<\infty.
    \end{align}
\end{definition}
\begin{definition} [Frequency synchronization] \label{def 3}
    The quaternionic Kuramoto model achieves frequency synchronization if for all $(n,m)\in \{1,...,N\}^2$
    \begin{align}
        \lim_{t\rightarrow \infty} |\dot{q}_n(t)-\dot{q}_m(t)|=0. 
    \end{align}
\end{definition}
\begin{definition} [Phase synchronization] \label{def 4}
    The quaternionic Kuramoto model achieves phase synchronization if for all $(n,m)\in \{1,...,N\}^2$
    \begin{align}
        \lim_{t\rightarrow \infty} |q_n(t)-q_m(t)|=0. 
    \end{align}
\end{definition}
\begin{definition} [Critical coupling strength] \label{def 5}
    For the quaternionic Kuramoto model, the critical coupling strength can be defined by
    \begin{align} \label{eq:lambda_c}
        \lambda_c:= \max_{n,m\in\{1,...N\}} |\omega_n-\omega_m|.
    \end{align}
\end{definition}

We pause to remark that the norm of $q_n$ can be expressed as
\begin{align}
    |q_n|=\sqrt{w_n^2+x_n^2+y_n^2+z_n^2}=\sqrt{det(\phi(q_n))}.
\end{align}
Also, the Definition \ref{def 3} is equivalent to the following condition
\begin{align}
    \lim_{t\rightarrow \infty} |\dot{q}_n(t)|=0, ~\mbox{for}~\mbox{all}~n\in\{1,...,N\},
\end{align}
since we assume the mean of natural frequencies is zero in \eqref{zero mean omega}. Finally, we want to emphasize that in real Kuramoto models, one can prove that the solution arrives at phase-locking synchronization if and only if it achieves frequency synchronization. However, this is not true in the quaternionic Kuramoto model, as we will emphasize in the following sections.
\\
\section{Strong coupling strength} \label{sec 3}
\begin{lemma} [Real part phase-locking] \label{thm 2}
    Let coupling strength $\lambda>\lambda_c/\sin(\delta)$ for some $0<\delta<\pi$. Assume that the solution $q(t)$ of system equations \eqref{main eq} satisfies initial condition $w_n(0)\in [0,\pi-\delta]$ for all $n\in\{1,...,N\}$. Then the solution achieves: 
    \begin{align}
        \sup_{t>0} \max_{n,m\in\{1,...N\}}|w_n(t)-w_m(t)|<\infty.
    \end{align}
\end{lemma}
\begin{proof}
    Firstly, we assume $w_n(0)\in (0,\pi-\delta)$ for all $n\in\{1,...,N\}$. For each pair $(n,m) \in \{1,...,N\}^2$, we claim that the phase difference between each pair for real parts is bounded as follows
    \begin{align} \label{bdd of w}
        |w_n(t)-w_m(t)| < \pi-\delta,
    \end{align}
    for all $t \geqslant 0$. Suppose, on the contrary, that \eqref{bdd of w} does not hold. It means that there exist a $t^*>0$, and a pair $(n,m)\in \{1,...N\}^2 $ such that
    \begin{align} \label{=pi-delta}
        |w_n(t^*)-w_m(t^*)|=\pi-\delta
    \end{align}
    and 
    \begin{align} \label{< pi-delta}
        \sup_{t\in[0,t^*-\epsilon]} |w_s(t)-w_l(t)|<\pi-\delta,
    \end{align}
    for each pair $(s,l) \in \{1,...,N\}^2$, for all $0<\epsilon \ll\ 1$.
    We may assume $w_n(t^*)-w_m(t^*)>0$ without loss of generality. Hence by means of \eqref{< pi-delta}, it is clear that $w'_n(t^*)-w'_m(t^*) \geqslant 0$. However, note that the real part difference dynamics of these two oscillators using system equations \eqref{main eq} can be written as 
    \begin{align*} \notag
        \dot{w}_n(t)-\dot{w}_m(t)=&\;\omega_n-\omega_m+\dfrac{\lambda}{N} \sum_{l=1}^N \Bigg(\sin(w_l-w_n)\cosh(v_{ln})\\ \notag
        &~~~~~~~~~~~~~~~~~~~~~~~~~-\sin(w_l-w_m)\cosh(v_{lm})\Bigg).\\
    \end{align*}
    Therefore, through \eqref{=pi-delta}, we obtain
    \begin{align} \notag
        \dot{w}_n(t^*)-\dot{w}_m(t^*)=&\;\omega_n-\omega_m+\dfrac{\lambda}{N} \sum_{l=1}^N \Bigg(\sin(w_l-w_m-(\pi-\delta))\cosh(v_{ln})\\ \notag
        &~~~~~~~~~~~~~~~~~~~~~~~~~-\sin(w_l-w_m)\cosh(v_{lm})\Bigg)\\ \notag
        \leqslant&\;\lambda_c+\dfrac{\lambda}{N} \sum_{l=1}^N \Bigg(\sin(w_l-w_m-(\pi-\delta))-\sin(w_l-w_m)\Bigg)\\ \notag
        \leqslant&\;\lambda_c-\lambda \sin(\delta)\\ \label{wn-wm<0}
        <& \;0.
    \end{align}
    A contradiction then verifies the claim. Finally, if the initial conditions satisfy $w_n(0)\in [0,\pi-\delta]$ for all $n\in\{1,...,N\}$, then inequality \eqref{wn-wm<0} still holds when $t=0$. Therefore, the proof is completed.
    \end{proof}
    
    Our next goal is to analyze the sufficient condition of phase and frequency synchronization for imaginary parts. In order to control the $v_{nm}$, we need to further restrict the quantity for $w_n-w_m$ for each pair $(n,m)\in\{1,...,N\}^2$. \\

\begin{lemma} [Imaginary parts phase and frequency synchronization] \label{p and f sync for Im}
    Let $q(t)$ be the solution of system equations \eqref{main eq}. Suppose that $\sup_{t>0} |w_n(t)-w_m(t)|<\pi/2-\delta_0$, for some $0<\delta_0\ll 1$, for each pair $(n,m)\in\{1,...,N\}^2$. Then the imaginary part of the solution arrives at phase and frequency synchronization
     \begin{align}
        \lim_{t\rightarrow \infty} |\mathbf{v}_n(t)-\mathbf{v}_m(t)|=0,
    \end{align}
    and
    \begin{align}
        \lim_{t\rightarrow \infty} |\dot{\mathbf{v}}_n(t)-\dot{\mathbf{v}}_m(t)|=0. 
    \end{align}
\end{lemma}

\begin{proof}
    Let us focus on the first component of $\mathbf{v}(t)$. Define 
    \begin{align}
        \mathbf{x}(t):=\max_{s,l\in\{1,...,N\}}|x_s(t)-x_l(t)|.
    \end{align}
    It is obvious that $\mathbf{x}(t):\mathbb{R}^+\rightarrow \mathbb{R}$ is a continuous and piecewise smooth function. For any $t>0$, there exists a pair $(n,m)\in\{1,...,N\}^2$ satisfying $\mathbf{x}(t)=x_n(t)-x_m(t)$. A straightforward calculation for this phase difference reveals 
    \begin{align*}
        \dfrac{d}{dt} \mathbf{x}&=\dot{x}_n-\dot{x}_m\\
        &=\dfrac{\lambda}{N}\sum_{l=1}^N \Bigg(\cos(w_l-w_n)\sinh(v_{ln})\dfrac{x_l-x_n}{v_{ln}}\\
        &~~~~~~~~~~~~~~~-\cos(w_l-w_m)\sinh(v_{lm})\dfrac{x_l-x_m}{v_{lm}}\Bigg)\\
        &\leqslant \dfrac{\lambda\cos(\frac{\pi}{2}-\delta_0)}{N}\sum_{l=1}^N \Bigg(\sinh(v_{ln})\dfrac{x_l-x_n}{v_{ln}}-\sinh(v_{lm})\dfrac{x_l-x_m}{v_{lm}}\Bigg)\\
        &\leqslant \dfrac{\lambda\cos(\frac{\pi}{2}-\delta_0)}{N}\sum_{l=1}^N \Bigg( (x_l-x_n)-(x_l-x_m)\Bigg)\\
        &\leqslant -\lambda \sin(\delta_0) (x_n-x_m).\\
        &=-\lambda\sin(\delta_0) \mathbf{x}.
    \end{align*}
    Thanks to the Gr{\"o}nwall inequality, then for any $t>0$ we obtain,
    \begin{align} \label{exponentially decreasing}
        \mathbf{x}(t)\leqslant \mathbf{x}(0) \exp(-\lambda\sin(\delta_0) t).
    \end{align}
    This means that the first component of imaginary parts arrives at phase and frequency synchronization simultaneously.
    Additionally, by applying the same method to other imaginary parts, we attain
    \begin{align}
        \lim_{t\rightarrow \infty} |\mathbf{v}_n(t)-\mathbf{v}_m(t)|=0,
    \end{align}
    and
    \begin{align}
        \lim_{t\rightarrow \infty} |\dot{\mathbf{v}}_n(t)-\dot{\mathbf{v}}_m(t)|=0. 
    \end{align}
    This completes the proof.
\end{proof}

\begin{theorem} [Full phase-locking] \label{q p locking}
    Consider the coupling strength and initial conditions described in Lemma \ref{thm 2}. Then the solution $q(t)$ achieves
    \begin{align}
        \sup_{t>0} \max_{n,m\in\{1,...N\}}|q_n(t)-q_m(t)|<\infty.
    \end{align}
\end{theorem}

\begin{proof}
    System equations \eqref{main eq} tell us that the sign of the cosine function plays an important role in determining whether the repulsion effect occurs in imaginary parts. Firstly, we claim that there exists $0<\delta_0 \ll 1$ such that 
    \begin{align}
        \limsup\limits_{t\rightarrow \infty}|w_n(t)-w_m(t)|\leqslant \dfrac{\pi}{2}-\delta_0,
    \end{align} for each pair $(n,m)\in\{1,...,N\}^2$. Recalling Lemma \ref{thm 2}, we immediately notice that $$\sup_{t>0}|w_n(t)-w_m(t)|<\pi-\delta \leqslant \pi/2-\delta_0$$ is true, if $\delta \geqslant \pi/2+\delta_0$. On the other hand, for $\delta<\pi/2-\delta_0$, combining the estimation \eqref{bdd of w} in Lemma \ref{thm 2} we obtain 
    \begin{align} \label{second claim}
        |w_{\mathrm{min}}(t)-w_{\mathrm{max}}(t)|<\pi-\delta,
    \end{align}
    for all $t>0$. Now suppose that there exists $t^*>0$ such that 
    \begin{align} \label{second claim}
        \dfrac{\pi}{2}-\delta_0<|w_{\mathrm{max}}(t^*)-w_{\mathrm{min}}(t^*)|<\pi-\delta,
    \end{align}
    where
    \begin{align}
    w_{min}(t^*)\leqslant w_l(t^*)\leqslant w_{max}(t^*),  
    \end{align}
    for all $l\in\{1,...,N\}$. Using the similar estimations established in the previous argument, denoting $w_{max}(t^*)-w_{min}(t^*)=\beta$ we obtain
    \begin{align*}
        \dot{w}_{max}(t^*)-\dot{w}_{min}(t^*)\leqslant&\;\lambda_c+\dfrac{\lambda}{N} \sum_{l=1}^N \Bigg(\sin(w_l(t^*)-w_{max}(t^*))-\sin(w_l(t^*)-w_{min}(t^*))\Bigg)\\
    \leqslant&\;\lambda_c+\max_{w^*\in[\pi/2-\delta_0,\beta]} \lambda\Bigg(\sin(w^*-\beta)-\sin(w^*)\Bigg)\\
    \leqslant&\;\lambda_c-\lambda\sin(\delta)    
    \end{align*}
    Since $\lambda\sin(\delta)-\lambda_c$ is strictly positive, we obtain
    \begin{align} 
        \sup_{t>T_c}\max_{n,m\in\{1,...,N\}}|w_n(t)-w_m(t)|<\dfrac{\pi}{2}-\delta_0,
    \end{align}
    where the critical time is
    \begin{align} \label{Big T}
           T_c=t^*+\dfrac{\beta-\frac{\pi}{2}+\delta_0} {\lambda\sin(\delta)-\lambda_c}.
    \end{align}

    By Lemma \ref{p and f sync for Im}, we obtain
    \begin{align}
        \sup_{t>0} \max_{n,m\in\{1,...N\}}|q_n(t)-q_m(t)|<\infty.
    \end{align}    
     \end{proof}
     
\begin{theorem} [Frequency synchronization] \label{F sync}
     Consider coupling strength $\lambda\sin(\delta)>\lambda_c$ for some $0<\delta<\pi$. Let $q(t)$ be the solution of system equations \eqref{main eq} satisfying initial condition $w_n(0)\in (0,\pi-\delta)$ for each $n\in\{1,...,N\}$. Then the solution arrives at frequency synchronization
     \begin{align}
        \lim_{t\rightarrow \infty} |\dot{q}_n(t)-\dot{q}_m(t)|=0,
    \end{align}
    for all $(n,m)\in\{1,...,N\}^2$.
\end{theorem}
\begin{proof}
    Based on Lemma \ref{p and f sync for Im} and Theorem \ref{q p locking}, we know that the imaginary part, $\mathbf{v}$, achieves phase and frequency synchronization. It remains to show that the real part also arrives at frequency synchronization. We consider a Lyapunov-like energy function as the following:
    \begin{align} \label{H function}
        \mathcal{H}(t):= \int_0^t \sum_{n=1}^N \dot{w}^2_n(s) \;ds.
    \end{align}
    We notice that $\mathcal{H}(t)$ is an increasing function with respect to $t$. The results proved in Lemma \ref{p and f sync for Im} and Theorem \ref{q p locking} imply that $\dot{w}_n(t)$ and $\ddot{w}_n(t)$ are bounded, for each $n\in\{1,...,N\}$. Hence, $\mathcal{H}'(t)$ is uniformly continuous. If $\mathcal{H}(t)$ is bounded by a constant that is independent of $t\in\mathbb{R}^+$, then 
    \begin{align}
        \lim\limits_{t\rightarrow \infty}\sum_{n=1}^N \dot{w}^2_n(t)=0.
    \end{align}
By means of system equations \eqref{main eq}, we arrive at
\begin{align*}
    \mathcal{H}(t)=&\int_0^t \sum_{n=1}^N \dot{w}^2_n(s) ds\\
   =&\int_0^t \left(\sum_{n=1}^N \omega_n \dot{w}_n(s) +\dfrac{\lambda}{N} \sum_{n=1}^N\sum_{m=1}^N \sin(w_m(s)-w_n(s))\cosh(v_{mn}(s)) \dot{w}_n(s)\right)\;ds\\
   =&\sum_{n=1}^N \omega_n(w_n(t)-w_n(0))\\
    &-\dfrac{\lambda}{N}\sum_{1\leqslant m<n \leqslant N} \int_0^t \sin(w_n(s)-w_m(s))\cosh(v_{mn}(s))(\dot{w}_n(s)-\dot{w}_m(s))\; ds\\
    \leqslant&\left|\sum_{n=1}^N \omega_n(w_n(t)-w_n(0))\right|+\left|\sum_{1\leqslant m<n \leqslant N}(\cos(w_n(s)-w_m(s))\cosh(v_{mn}(s))\Big|_{s=0}^{s=t}\right|\\
    &+\left|\dfrac{\lambda}{N}\sum_{1\leqslant m<n \leqslant N} \int_0^t \cos(w_n(s)-w_m(s))\sinh(v_{mn}(s))(\dot{v}_{mn}(s)) \; ds\right|\\
    =:& ~\mathcal{I}(t)+\mathcal{II}(t)+\mathcal{III}(t).
\end{align*}
We want to demonstrate that $\limsup_{t\rightarrow \infty} \mathcal{I+II+III}$ is bounded by a constant. It is difficult to directly observe that $\limsup_{t\rightarrow \infty}\mathcal{II}(t)$ is bounded since cosine is bounded by $1$ and $\limsup_{t\rightarrow \infty} \cosh(v_{mn}(t))=\cosh(1)$ due to Lemma \ref{p and f sync for Im} and Theorem \ref{q p locking}. Recalling the zero mean assumption for natural frequencies in \eqref{zero mean omega}, then we can obtain the bounds for $\mathcal{I}(t)$ by
\begin{align*}
    \mathcal{I}(t)\leqslant \sum_{n=2}^N \omega_n\left|(w_n(t)-w_1(t))\right|+\sum_{n=2}^N \omega_n\left|(w_n(0)-w_1(0))\right|.
\end{align*}
Using Lemma \ref{thm 2}, it is obvious that $\limsup_{t\rightarrow \infty}\mathcal{I}(t)$ is bounded.

It remains to verify $\limsup_{t\rightarrow \infty} \mathcal{III}(t)$ is bounded. Recalling the definition of $v_{mn}$ in system equations \eqref{main eq}, Lemma \ref{p and f sync for Im} and Theorem \ref{q p locking} and the definition of $T_c$ in \eqref{Big T}, we observe that there exists a $L>0$ and $T>T_c>0$ such that 
\begin{align} \label{v_mm dot is bdd}
    \sup_{t>T}|\dot{v}_{mn}(t)|<L,
\end{align}
for all $(n,m)\in\{1,...,N\}^2$. Therefore, combining \eqref{v_mm dot is bdd} and exploiting Lemma \ref{p and f sync for Im} and Theorem \ref{q p locking}, we obtain
\begin{align}
    \limsup_{t\rightarrow \infty}\mathcal{III}(t)\leqslant & \dfrac{\lambda}{N} \sum_{1\leqslant m<n\leqslant N}\int_0^{T} \left|\sinh(v_{mn}(s))\dot{v}_{mn}(s)\right|\; ds\\ \label{integration of sinh}
    &+\dfrac{\lambda L}{N} \sum_{1\leqslant m<n\leqslant N}\int_{T}^\infty \sinh(\left|v_{mn}(s)\right|)\; ds.
\end{align}
Due to inequality \eqref{exponentially decreasing}, we realize that $v_{mn}(t)$ is decreasing to zero exponentially, for all pair $(m,n)\in\{1,...,N\}^2$. Therefore, this implies that the second integration \eqref{integration of sinh} is bounded, and hence $\limsup_{t \rightarrow \infty}\mathcal{III}(t)$ is bounded. 

To recapitulate, we have proben that $\mathcal{H}(t)$ is bounded by a constant which is independent of $t\in\mathbb{R}^+$, so the solution $q(t)$ of system equations \eqref{main eq} arrives at frequency synchronization. This completes the proof.
\end{proof}

\begin{theorem} [Phase synchronization with identical nature frequencies] \label{3.5}
    Let $\boldsymbol{\omega}\equiv \mathbf{0}$. Assume that the solution of $q(t)$ of system equations \eqref{main eq} satisfies initial condition $w_n(0)\in(0,\pi-\delta)$ for some $\delta>0$. Then the solution arrives at phase synchronization.
\end{theorem}
\begin{proof}
    Using Lemma \ref{thm 2} and Lemma \ref{p and f sync for Im}, we know that all of the imaginary parts tend to zero exponentially in \eqref{exponentially decreasing}. Also, previous results showed in Theorem \ref{q p locking} and Theorem \ref{F sync} demonstrate that the difference of real parts are bounded by $\pi/2$ and the real part achieves frequency synchronization. Combining these results and considering the limit equation of the real part in system equations \eqref{main eq}, we obtain phase synchronization. This completes the proof.
\end{proof}


\subsection{Numerical results}
\label{subsec:numerical}
In this subsection, we provide some numerical results to support our analytical conclusions. 

In Fig.~\ref{fig:strong_q} and Fig.~\ref{fig:strong_q_dot}, we demonstrate a case of $N=5$ oscillators exhibiting full phase-locking and frequency synchronization in the quaternionic Kuramoto model, predicted by Theorem \ref{q p locking} and Theorem \ref{F sync}, respectively. In particular, we choose the natural frequencies $\boldsymbol{\omega} = (0.66, 0.10, -0.29, -0.34, -0.12)
$ so that $\lambda_c=1$ by \eqref{def 5}. Also, we choose coupling strength $\lambda=1.1$ and $\delta=\pi/2$, so that $\lambda>\lambda_c/\sin(\delta)$, such that the coupling strength is strong. The initial conditions are chosen as follows. For the real part, we choose $\vec{w}(0) = (0.14, 1.52, 0.36, 0.96, 0.15)$, all in $(0,\pi-\delta)$, so that the premise for Theorem \ref{q p locking} is satisfied. For the imaginary parts, we choose  $\vec{x}(0)=(0.86, 0.21, 0.36, 1.22, 0.87)$, $\vec{y}(0)=(0.09, 1.51, 0.15, 0.05, 1.27)$, and $\vec{z}(0)=(0.69, 0.22, 0.04, 0.12, 1.35).$ 

In Fig.~\ref{fig:strong_q_omega_zero}, we demonstrate a case of $N=5$ oscillators with identical natural frequencies that exhibit phase synchronization, predicted by Theorem \ref{3.5}. In particular, the natural frequencies $\boldsymbol{\omega} = (0,0,0,0,0)$ by the assumption \eqref{zero mean omega}. The initial conditions are chosen the same as above.

\begin{figure}[H] 
\centering
\includegraphics[width=8cm]{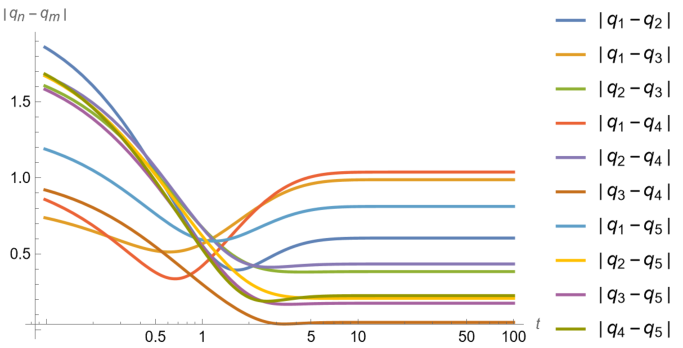}
\caption{This is the time evolution of $\{|q_n(t)-q_m(t)|\}_{1\leq m<n\leq N}$ for 5 oscillators ($N=5$) in the case of strong coupling strength. The parameters and initial conditions are given in the subsection \ref{subsec:numerical}. We can observe full phase-locking synchronization in this case, which is consistent with Theorem \ref{q p locking}.}

\label{fig:strong_q}
\end{figure}

\begin{figure}[H] 
\centering
\includegraphics[width=8cm]{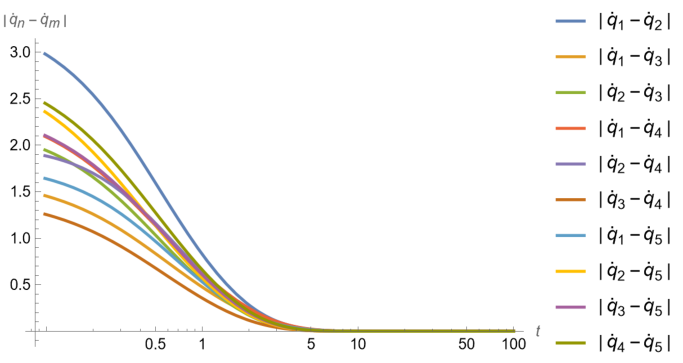}
\caption{This is the time evolution of $\{|\dot{q}_n(t)-\dot{q}_m(t)|\}_{1\leq m<n\leq N}$ for 5 oscillators ($N=5$) in the case of strong coupling strength. The parameters and initial conditions are given in the subsection \ref{subsec:numerical}. We can observe frequency synchronization in this case, which is consistent with Theorem \ref{F sync}.}
\label{fig:strong_q_dot}
\end{figure}

\begin{figure}[H] 
\centering
\includegraphics[width=8cm]{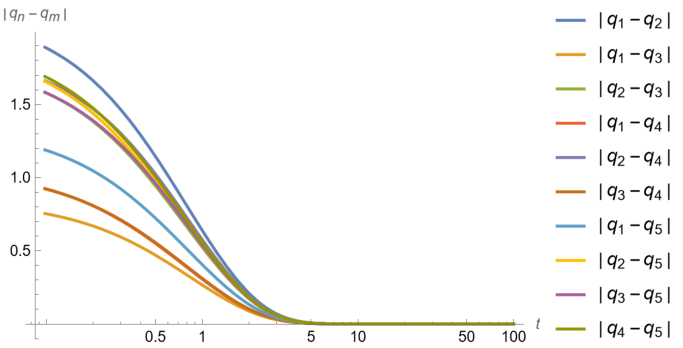}
\caption{This is the time evolution of $\{|q_n(t)-q_m(t)|\}_{1\leq m<n\leq N}$ for 5 oscillators ($N=5$) with identical natural frequencies ($\boldsymbol{\omega}=\boldsymbol{0}$) in the case of strong coupling strength. The parameters and initial conditions are given in the subsection \ref{subsec:numerical}. We can observe phase synchronization in this case, which is consistent with Theorem \ref{3.5}.}
\label{fig:strong_q_omega_zero}
\end{figure}

\section{Weak coupling strength} \label{sec 4}
\subsection{Two oscillators} \label{sec 4.1}
This subsection will analyze the phase-locking synchronization with the weak coupling strength. We first illustrate this with the two oscillators ($N=2$) system. Consider two oscillators with $\omega:=2\omega_1>0$ and the weak coupling strength $\lambda<\omega$. Define $q=(w,x,y,z):=q_1-q_2$. Recalling the definition of $v_{12}$ in \eqref{def of v_mn}, the system equations \eqref{main eq} may be written as
\begin{equation} \label{N=2}
\left\{ \begin{aligned}
    \dot{w}=& \;\omega-\lambda\sin(w)\cosh(v_{12}),\\
    \dot{v}_{12}=&~~~-\lambda\cos(w)\sinh(v_{12}).
\end{aligned}
\right.
\end{equation}
We pause to remark two things: firstly, it is impossible to arrive at phase-locking synchronization in the real $N=2$ Kuramoto model with a weak coupling ($\lambda<\omega$) because there are no fixed points. Secondly, this system can be reduced to the complexified Kuramoto model \cite{thumler2023synchrony} if we restrict $y_1\equiv y_2$ and $z_1 \equiv z_2$.

However, fixed points exist for the quaternionic Kuramoto model, such as $(w,v_{12})=(\pi/2,\cosh^{-1}(\omega/\lambda))$, where $\cosh^{-1}(x):=\ln\left(x+\sqrt{x^2-1}\right)$ on $1\leqslant x<\infty$. We prove this in the following theorem. To simplify the notation, let $v:=v_{12}$, $\gamma:=\omega/\lambda>1$, and $\alpha:=\cosh^{-1}(\omega/\lambda)>0$. The system of equations \eqref{N=2} can be rewritten as follows.

\begin{theorem} \label{weak N=2}
    Consider 
    \begin{equation} \label{n=2}
\left\{ \begin{aligned}
    \dot{w}=& \;\omega-\lambda\sin(w)\cosh(v),\\
    \dot{v}=&~~~-\lambda\cos(w)\sinh(v),
\end{aligned}
\right.
\end{equation}
    with weak coupling strength $\lambda<\omega$. Then, every equilibrium point $(w_k,v_k)=(2k\pi+\pi/2,\alpha)$, where $k\in\mathbb{Z}$ is Lyapunov stable but not asymptotically stable, and hence for any solution with initial condition close enough to these equilibria is a periodic orbit.
\end{theorem}

We also notice two things here: firstly, the equilibria formed in this system are not attractors, since the closed contours form a set of dense and closed ring contours near each of them, resembling a tree's growth rings. Secondly, the signs of $\dot{w}$ and $\dot{v}$ depend on the relative position of the equilibrium points.

Because of the possible interactions between the real frequency $w$ and the imaginary frequencies $x,y,z$ that comprise $v$, we cannot expect classical methods in the real Kuramoto model to work. Before further discussions, we conveniently separate the domain near each equilibrium point into four ``quadrants" in the $(w,v)$-plane as
    \begin{align} \label{Q1k}
        Q_{1k} &:= \left\lbrace (w,v) ~\bigg\vert~ w_k<w<w_k+\dfrac{\pi}{2}, ~ \alpha < v < \infty \right\rbrace ;\\ \label{Q2k}
        Q_{2k} &:= \left\lbrace (w,v) ~\bigg\vert~ w_k-\dfrac{\pi}{2} < w < w_k, ~ \alpha < v < \infty \right\rbrace; \\ \label{Q3k}
        Q_{3k} &:= \left\lbrace (w,v) ~\bigg\vert~ w_k-\dfrac{\pi}{2} < w < w_k, ~ 0 < v < \alpha \right\rbrace; \\ \label{Q4k}
        Q_{4k} &:= \left\lbrace (w,v) ~\bigg\vert~ w_k<w<w_k+\dfrac{\pi}{2}, ~ 0 < v < \alpha \right\rbrace.
    \end{align}

Now we define the lemmas that will lead us to the solutions of the contour trajectories around the fixed points. We will prove Theorem \ref{weak N=2} completely in \ref{4.1proof}.

\begin{lemma}  [$\delta/n$ criterion] \label{delta n criterion}
    Assume the solution $(w(t),v(t))$ of the system equations \eqref{n=2} satisfies the initial condition
    \begin{align} \label{n=2 initial}
        (w(0),v(0))\in Q_{2k}\bigcap \left\lbrace (w,v) ~\bigg\vert~\omega-\lambda \sin(w)\cosh(v)=0 \right\rbrace.
    \end{align}
    Then, there exists a finite time $T>0$ such that 
    \begin{align} \label{wT vT}
        (w(T),v(T))\in  \left\lbrace (w,v) ~\bigg\vert~w_k-\dfrac{\pi}{2}<w<w_k,v=\alpha \right\rbrace.
    \end{align}
\end{lemma}
\begin{proof}
Define 
\begin{align} \label{Q tilde}
    \tilde{Q}_{2k}:=Q_{2k}\bigcap \left\{ (w,v) ~\bigg\vert~ \omega-\lambda\sin(w)\cosh(v) \geqslant 0, ~w\geqslant w(0) \right\}.
\end{align}
We illuminate this in Fig.~\ref{fig:M2}, where the cyan curve represents $\omega-\lambda\sin(w)\cosh(v)=0$. We observe that the solution $(w(t),v(t))$ with the initial condition \eqref{n=2 initial} keeps staying at the right-hand side of the vertical line $w=w(0)$, since $\dot{w}\geqslant 0$. Also, the solution cannot cross the cyan curve $\omega-\lambda\sin(w)\cosh(v)=0$, since $\dot{w}=0$ and $\dot{v}<0$. 
If the conclusion \eqref{wT vT} is not true, then the solution $(w(t),v(t))$ will stay at the region $\tilde{Q}_{2k}$ for any finite time $t>0$. By Poincar{\'e}–Bendixson theorem, this implies the trajectory satisfies 
\begin{align} \label{Q2k condition}
    (w(t),v(t))\in \tilde{Q}_{2k} ~\mbox{when}~ 0<t<\infty,
\end{align}
and 
\begin{align} \label{limit behavior}
    \lim\limits_{t\rightarrow \infty} (w(t),v(t))=(w_k,\alpha).
\end{align}

On the other hand, consider a $\delta>0$, such that $w(0)+\delta<w_k.$ Let $P_n=(\varpi_0,\varpi_1,...,\varpi_n)$ be a partition of the closed interval $[w(0),w(0)+\delta]$ and $\varpi_l=w(0)+l\delta/n$ for each $l\in\{0,1,...,n\}$. Our goal here is to construct an infinite-layer ``$\delta/n$ buffer" to slow down the solution and avoid it approaching the equilibrium point $(w_k,\alpha)$. In the region $\tilde{Q}_{2k}\bigcap \{ \varpi_{l-1}<w<\varpi_{l}\}$, for $l=1,...n$, we notice that 
\begin{align} \label{dot w}
    \dot{w}=& \;\omega-\lambda\sin(w)\cosh(v) \leqslant \omega-\lambda\sin(\varpi_{l-1})\cosh(\alpha),
\end{align}
and
\begin{align} \label{dot v}
    |\dot{v}|=&\lambda\cos(w)\sinh(v) \geqslant \lambda \cos(\varpi_{l})\sinh(\alpha).
\end{align}
Because of \eqref{Q2k condition} and \eqref{limit behavior}, there exists a $\tilde{t}$ such that $w(\tilde{t})=\varpi_n$ and $v(\tilde{t})>\alpha$. Also, we obtain
\begin{align} \label{v0-alpha}
    v_0-\alpha > \int_0^{\tilde{t}} -\dot{v} \;dt=\int_{w(0)}^{w(\tilde{t})} -\dfrac{\dot{v}}{\dot{w}} \;dw=\sum_{l=0}^{n}\int_{\varpi_l}^{\varpi_{l+1}} -\dfrac{\dot{v}}{\dot{w}} \;dw.
\end{align}
By means of \eqref{dot w} and \eqref{dot v}, we arrive at 
\begin{align} \notag
    &\sum_{l=0}^{n}\int_{\varpi_l}^{\varpi_{l+1}} -\dfrac{\dot{v}}{\dot{w}} \;dw \\ \notag
    \geqslant &\sum_{l=0}^{n} \dfrac{\lambda \cos(\varpi_l)\sinh(\alpha)}{\omega-\lambda\sin(\varpi_{l-1})\cosh(\alpha)}\cdot\dfrac{\delta}{n}
    =
    \tanh(\alpha)
    \sum_{l=0}^{n} \dfrac{ \cos(\varpi_l)}{1-\sin(\varpi_{l-1})}\cdot\dfrac{\delta}{n}\\  \label{riemann sum}
    =&
    \tanh(\alpha)
    \left(
    \cos\left(\dfrac{\delta}{n}\right) \sum_{l=0}^{n} \dfrac{ \cos(\varpi_{l-1})}{1-\sin(\varpi_{l-1})}\cdot\dfrac{\delta}{n}-\sin\left(\dfrac{\delta}{n}\right) \sum_{l=0}^{n} \dfrac{ \sin(\varpi_{l-1})}{1-\sin(\varpi_{l-1})}\cdot\dfrac{\delta}{n}
    \right)
    .
\end{align}
We notice that inequality \eqref{riemann sum} holds for all $n\in \mathbb{N}$. Therefore, combining \eqref{v0-alpha} and \eqref{riemann sum} results in
\begin{align} \notag
    &v_0-\alpha \\ \notag
    \geqslant &\tanh(\alpha)\cdot
    \lim_{n\rightarrow \infty}\left\{ \cos\left(\dfrac{\delta}{n}\right) \sum_{l=0}^{n} \dfrac{ \cos(\varpi_{l-1})}{1-\sin(\varpi_{l-1})}\cdot\dfrac{\delta}{n}-\sin\left(\dfrac{\delta}{n}\right) \sum_{l=0}^{n} \dfrac{ \sin(\varpi_{l-1})}{1-\sin(\varpi_{l-1})}\cdot\dfrac{\delta}{n} \right\} \\ \notag
    =&
    \tanh(\alpha)\cdot
    \int_{w(0)}^{w(0)+\delta} \dfrac{\cos(w)}{1-\sin(w)} \;dw 
    = 
    \tanh(\alpha)\cdot
    \ln\left( \dfrac{1-\sin(w(0))}{1-\sin(w(0)+\delta)} \right),
\end{align}
which is bounded. This is a contradiction of
\begin{align}
    \ln\left( \dfrac{1-\sin(w(0))}{1-\sin(w(0)+\delta)} \right)\rightarrow \infty,
\end{align}
as $\delta\rightarrow w_k-w(0)$. This contradiction thus asserts that
there exists a finite time $T>0$ such that 
    \begin{align} 
        (w(T),v(T))\in  \left\lbrace (w,v) ~\bigg\vert~w_k-\frac{\pi}{2}<w<w_k,v=\alpha \right\rbrace.
    \end{align}
\end{proof}

Now we prove that the trajectory is symmetric about the vertical line $w=w_k$.

\begin{lemma}  [Symmetric property] \label{sym prop}
    Let $(w(t),v(t))$ be the solution of system equations \eqref{n=2} with the initial condition 
    \begin{align} \label{IC symmetric property}
         (w(0),v(0)) \in\bigcup_{i=1}^{4} cl (Q_{ik}).
    \end{align}
    Assume that there exists a $T>0$, such that $w(T)=w_k$ and $v(T)>0$, for some $k\in\mathbb{Z}$. Then the trajectory $(w(t),v(t))$ must exist until at least $t=2T$, and is hence symmetric about the vertical line $w=w_k$, when $0<t<2T$.  
\end{lemma}
\begin{proof}
    Let
        \begin{align} \label{W}
            W(t) &:=2 w_k-w(T-t),\\ \label{V}
            V(t) &:= v(T-t),
        \end{align}
        for $0\leqslant t\leqslant T$. Notice that $W(0)=w_k=w(T)$ and $V(0)=v(T)$. Also,
        \begin{align*}
            \dot{W}(t)
            &=
            \dot{w}(T-t)
            =
            \omega - \lambda \sin(w(T-t)) \cosh(v(T-t))
            \\
            &=
            \omega - \lambda
            \sin(W(t)) \cosh(V(t));
            \\
            \dot{V}(t)
            &=
            -\dot{v}(T-t)
            =
            \lambda \cos(w(T-t)) \sinh(v(T-t))
            \\
            &=
            -\lambda \cos(W(t)) \sinh(V(t))
            ,
        \end{align*}
        such that $(W(t),V(t))$ also satisfies the original system equations \eqref{n=2}.\\
        \\
        Through the existence and uniqueness of the solution, we have  
        \begin{align*}
            W(t) &= w(t+T); \\
            V(t) &= v(t+T)
        \end{align*}
        for $t\in[0,T]$. We pause to remark that the dynamical system must continually evolve until at least $t=2T$. Otherwise, if there is a $t_0\in (T,2T)$ such that $w$ or $v$ goes to infinity by symmetry up to $t_0-\epsilon$, there is a contradiction through the continuity argument. Combining \eqref{W} and \eqref{V} implies that 
        \begin{align*}
            w(T-t) + w(T+t) &= 2 w_k,\\
            v(T-t) - v(T+t) &= 0,
        \end{align*}
        so the solution must be symmetric with respect to the vertical line $w=w_k$.
\end{proof}
\begin{lemma}  [Deceleration/Acceleration region]\label{DA region}
   Given $k\in\mathbb{Z}$, let a Lyapunov function be defined as follows.
    $$
    L_k(w,v)
    =
   \frac{1}{2}  \left(\left(w-w_k \right)^2+\left(v-\alpha\right)^2\right).
    $$ There exists an $\epsilon>0$, only depending on $\alpha$, such that for every $k\in\mathbb{Z}$, the following function
    \begin{align*}
         F_k(w,v)
        :=\dfrac{1}{\lambda}\dfrac{dL_k}{dt}=
        \left(w-w_k\right)(\gamma-\sin(w)\cosh(v))
        -
        (v-\alpha) \cos(w) \sinh(v)
    \end{align*}
    is negative on the set $S^-_k$ and is positive on the set $S^+_k$, where
    \begin{align*}
        S^-_k &:= \left\lbrace (w,v) ~\bigg\vert~ w_k-\dfrac{\pi}{2} < w < w_k, ~ (1-\epsilon)\alpha < v < (1+\epsilon)\alpha \right\rbrace ,\\
        S^+_k &:= \left\lbrace (w,v) ~\bigg\vert~ w_k<w<w_k+\dfrac{\pi}{2}, ~ (1-\epsilon)\alpha < v < (1+\epsilon)\alpha \right\rbrace .
    \end{align*}
\end{lemma}

\begin{proof}
    Recall that $w_k=\frac{\pi}{2}+2k\pi$. Since both the function $F_k$ and the sets $S_k^-,~S_k^+$ are periodic $2\pi$, it suffices to focus on the case $k=0$. For brevity, we write $F:=F_0$.
    
    To determine the sign of $F$, we argue through the signs of its partial derivatives $F_w$ and $F_{ww}$. A summary of the signs of $F$, $F_w$ and $F_{ww}$ are given in Fig.~\ref{fig:M1}.
    
    First, we check that, for all $v$,
    \begin{align*}
        F(0,v)
        &=
        -\dfrac{\pi}{2}\gamma-(v-\alpha)\sinh(v)
        ,\\
        F\left(\dfrac{\pi}{2},v\right) 
        &= 0
        ,\\
        F\left(\pi,v\right) 
        ~&= 
        \dfrac{\pi}{2}\gamma+(v-\alpha)\sinh(v)
        .
    \end{align*}
    Since the function $h_1(v):=(v-\alpha)\sinh(v)$ is continuous and $h_1(\alpha)=0$, there exists an $\epsilon_1>0$ such that for \emph{all} $(1-\epsilon_1)\alpha < v < (1+\epsilon_1)\alpha$, we have $|h_1(v)| < \pi\gamma/2$. Therefore, whenever $(1-\epsilon_1)\alpha < v < (1+\epsilon_1)\alpha$, we have $F(0,v)<0$ and $F(\pi,v)>0$.
    
    Second, we look at the partial derivative $F_w$. A straightforward calculation yields that
    \begin{align*}
        F_w(w,v) &:= \dfrac{\partial}{\partial w} F(w,v)
        \\
        &=
        \gamma-\sin(w)\cosh(v)
        -
        \left(w-\dfrac{\pi}{2}\right) \cos(w)\cosh(v)
        +
        (v-\alpha) \sin(w) \sinh(v)
        .
    \end{align*}
    Hence, we have
    \begin{align*}
        F_w\left(0,v\right)
        ~&=
        F_w\left(\pi,v\right)
        =
        \gamma + \dfrac{\pi}{2} \cosh(v)
        \geqslant
        \gamma + \dfrac{\pi}{2}
        >
        0
        ,\\
        F_w\left(\dfrac{\pi}{2},v\right)
        &=
        \gamma - \cosh(v) + (v-\alpha) \sinh(v)
        .
    \end{align*}
    We notice that $F_w\left(\pi/2,\alpha\right)=0$. To determine the sign of $F_w\left(\pi/2,v\right)$ for $v$ close to $\alpha$, we consider the derivative of $F_w\left(\pi/2,v\right)$ w.r.t. $v$. Equivalently,
    \begin{align*}
        F_{wv}\left(\dfrac{\pi}{2},v\right)
        :=
        \dfrac{\partial^2}{\partial v\partial w} F(w,v) \bigg\vert_{w=\frac{\pi}{2}}
        =
        (v-\alpha) \cosh (v)
        .
    \end{align*}
    It is now clear that $F_{wv}(\pi/2,v)<0$ when $v<\alpha$ and that $F_{wv}(\pi/2,v)>0$ when $v>\alpha$. Hence, $F_{w}(\pi/2,v)$ is minimized at and only at $v=\alpha$. That is,
    \begin{align*}
        F_w\left(\dfrac{\pi}{2},v\right)
        \geqslant
        F_w\left(\dfrac{\pi}{2},\alpha\right)
        = 
        0
        .
    \end{align*}

    Third, we look at the second-order partial derivative $F_{ww}$. A straightforward calculation results in
    \begin{align}
        F_{ww}(w,v)
        &:=
        \dfrac{\partial^2}{\partial w^2} F(w,v)
        \nonumber\\
        &=
        -2\cos(w)\cosh(v)
        +
        \left(w-\dfrac{\pi}{2}\right) \sin(w)\cosh(v)
        +
        (v-\alpha) \cos(w) \sinh(v)
        \nonumber\\
        &=
        -\cos(w) \cdot
        \left(
        2 \cosh(v) - (v-\alpha) \sinh(v)
        \right)
        -
        \left(\dfrac{\pi}{2}-w\right) \sin(w)\cosh(v)
        .
        \label{eq:Fww}
    \end{align}
    Let us first restrict $w$ to be in the open interval $(0,\pi/2)$. Then, $\cos(w)>0$, $\pi/2-w>0$, $\sin(w)>0$ and $\cosh(v)>0$, so as long as $v$ allows the function $h_2(v):=2 \cosh(v) - (v-\alpha) \sinh(v)$ to be positive, then from \eqref{eq:Fww} we have $F_{ww}<0$. Since $h_2(\alpha)=2\gamma>0$, by continuity of $h_2$, there exists an $\epsilon_2>0$ such that $h_2(v)>0$ for \emph{all} $(1-\epsilon_2)\alpha<v<(1+\epsilon_2)\alpha$. 
    
    Now, let us restrict $w$ to be in the open interval $(\pi/2,\pi)$. Then, $\cos(w)<0$, $\pi/2-w<0$, $\sin(w)>0$ and $\cosh(v)>0$, so as long as $v$ allows $h_2(v)$ to be positive, from \eqref{eq:Fww} we have $F_{ww}>0$. Thus the same $\epsilon_2$ works.
    
    Let $\epsilon=\min\{\epsilon_1,\epsilon_2\}$. 
    Define two sets in the $(w,v)$-plane around the equilibrium point $(\pi/2,\alpha)$:
    \begin{align*}
        S^{-}_0 &:= \left\lbrace (w,v) ~\bigg\vert~ 0<w<\dfrac{\pi}{2}, ~ (1-\epsilon)\alpha < v < (1+\epsilon)\alpha \right\rbrace ;\\
        S^{+}_0 &:= \left\lbrace (w,v) ~\bigg\vert~ \dfrac{\pi}{2}<w<\pi, ~ (1-\epsilon)\alpha < v < (1+\epsilon)\alpha \right\rbrace .
    \end{align*}
    For reasons that will become clear in the following argument, we shall call $S^-_0$ the ``deceleration region" and $S^+_0$ the ``acceleration region" for convenience. We illuminate this in Fig.~\ref{fig:M1}, where the ``deceleration region" is the red region, and the ``acceleration region" is the green region.

    Fix any $v\in((1-\epsilon)\alpha,(1+\epsilon)\alpha)$, which defines a horizontal slice through the two regions. Since $F_{ww}<0$ on $S^-_0$, the one-variable function $F_w(w,v)$ is strictly decreasing. Because $F_w(0,v)>0$ and $F_w(\pi/2,v)\geqslant 0$, we must have
    \begin{align*}
        F_w(w,v) > 0 ~\text{for}~ w \in \left(0,\dfrac{\pi}{2}\right)
        .
    \end{align*}
    That is, the one-variable function $F(w,v)$ is strictly increasing. Since $F(0,v)<0$ and $F(\pi/2,v)=0$, we then have
    \begin{align*}
        F(w,v) < 0 ~\text{for}~ w \in \left(0,\dfrac{\pi}{2}\right)
        .
    \end{align*}
    Since $v$ was arbitrarily chosen in $((1-\epsilon)\alpha,(1+\epsilon)\alpha)$, we obtain that
    \begin{align*}
        F(w,v) < 0 ~\text{on}~ S^-_0
        .
    \end{align*}
    An analogous argument lets us conclude that
    \begin{align*}
        F(w,v) > 0 ~\text{on}~ S^+_0
        .
    \end{align*}
    
    \begin{figure}[h]
        \centering
        \begin{tikzpicture}
            \filldraw[color=green!25, fill=green!25, very thick] (0,-1) rectangle (2,1);
            \filldraw[color=red!25, fill=red!25, very thick] (-2,-1) rectangle (0,1);
            \draw[-] (-2.5,0) -- (2.5,0) node[anchor=west] {$v=\cosh^{-1}(\omega/\lambda)$};
            \draw[-] (-2.5,1) -- (2.5,1) node[anchor=west] {$v=(1+\epsilon)\cosh^{-1}(\omega/\lambda)$};
            \draw[-] (-2.5,-1) -- (2.5,-1) node[anchor=west] {$v=(1-\epsilon)\cosh^{-1}(\omega/\lambda)$};
            \draw[-,blue,thick] (0,-1.5) -- (0,-1/16) node[anchor=south,black] {$w=\pi/2$} -- (0,1.5);
            \draw[-,cyan,thick] (-2,-1.5) -- (-2,1/64) node[anchor=south,black] {$w=0$} -- (-2,1.5);
            \draw[-] (2,-1.5) -- (2,3/64) node[anchor=south,black] {$w=\pi$} -- (2,1.5);
            \draw[->,color=blue,thick] (0,-2.25) node[color=blue,anchor=north]{$F\vert_{w=\pi/2}=0$} -- (0,-1.75);
            \draw[->,color=cyan,thick] (-2,-2.25) node[color=cyan,anchor=north]{$F\vert_{w=0}<0$} -- (-2,-1.75);
            \draw[->,color=cyan,thick] (2,-2.25) node[color=cyan,anchor=north]{$F\vert_{w=\pi}>0$} -- (2,-1.75);
            \draw[->,color=blue,thick] (0,2.25) node[color=blue,anchor=south]{$F_w\vert_{w=\pi/2}\geq 0$} -- (0,1.75);
            \draw[->,color=cyan,thick] (-2,2.25) node[color=cyan,anchor=south]{$F_w\vert_{w=0}>0$} -- (-2,1.75);
            \draw[->,color=cyan,thick] (2,2.25) node[color=cyan,anchor=south]{$F_w\vert_{w=\pi}>0$} -- (2,1.75);
            \draw[color=purple!100,thick] (-1,0) node[color=red,anchor=north]{$F_{ww}<0$};
            \draw[color=teal!100,thick] (1,0) node[color=teal,anchor=north]{$F_{ww}>0$};
        \end{tikzpicture}
    \caption{This plot shows the ``deceleration strip" (red) and ``acceleration strip" (green) of the Lyapunov function $L$ in a neighborhood of the equilibrium point $(\pi/2,\cosh^{-1}(\omega/\lambda))$.} 
    \label{fig:M1} 
    \end{figure}
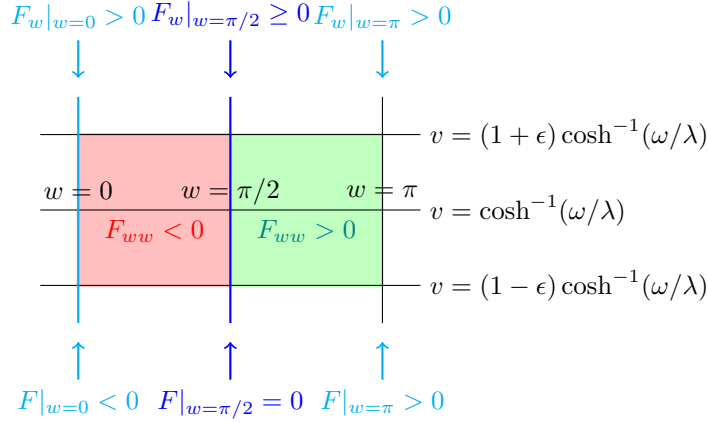
\end{proof}

Now we are ready to prove Theorem \ref{weak N=2}.

\begin{proof} \label{4.1proof}
    First, we focus on the equilibrium point $(w_0,v_0)=(\pi/2,\alpha)$, as the analysis for other equilibrium points $(w_k,v_k)$ are the same via periodicity.

   Consider the solution $(w(t),v(t))$ satisfying the system equations \eqref{n=2} with the initial condition 
    \begin{align} \label{eq:init_cond}
         w(0) = \dfrac{\pi}{2}
         ,~
        v(0)>\alpha  . 
    \end{align}
We observe that the trajectory $(w(t),v(t))$ must first enter the second quadrant $Q_{20}$ defined in \eqref{Q2k}, since $\dot{w}(0)<0$ and $\dot{v}(0)=0$, and hence fall in $\{ v<v(0)\}$ when $0<t\ll 1$, since $\dot{v}<0$ in $Q_{20}$.
We claim that the trajectory $(w(t),v(t))$ must hit the cyan curve, $\omega-\lambda\sin(w)\cosh(v)=0$, in Fig.~\ref{fig:M2}. To prove this, we suppose otherwise that
\begin{align}
    (w(0),v(0))\in \bar{Q}_{20}:=Q_{20}\bigcap \left\{ (w,v) ~\bigg\vert~ \omega-\lambda\sin(w)\cosh(v) \leqslant 0, ~v\leqslant v(0) \right\}.
\end{align}
By Poincar{\'e}–Bendixson theorem, either there exists a finite time $T>0$ such that $w(T)=\pi/2$ and $v(T)>\alpha$ or the asymptotic behaviour of the trajectory satisfies
\begin{align} \label{Q2k bar condition}
    (w(t),v(t))\in \bar{Q}_{20} ~\mbox{when}~ 0<t<\infty,
\end{align}
and 
\begin{align} \label{asymptotic behavior}
    \lim\limits_{t\rightarrow \infty} (w(t),v(t))=\left(\dfrac{\pi}{2},\alpha\right).
\end{align}
However, it is clear that the trajectory cannot land $\{ w=\frac{\pi}{2}, \alpha<v<v(0) \}\subset \bar{Q}_{20}$, since $\dot{w}<0$ in $\bar{Q}_{20}$. Because $\dot{w}<0$ in $\bar{Q}_{20}$, \eqref{Q2k bar condition} and \eqref{asymptotic behavior} are impossible to be true simultaneously. A contradiction then proves the claim. 

    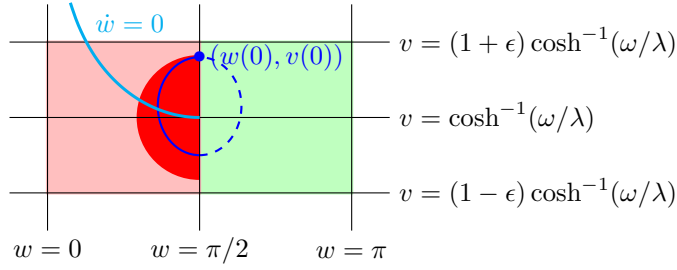
\begin{figure}\label{trajectory}
        \centering 
        \begin{tikzpicture}
            \filldraw[color=green!25, fill=green!25, very thick] (0,-1) rectangle (2,1);
            \filldraw[color=red!25, fill=red!25, very thick] (-2,-1) rectangle (0,1);
            \filldraw[color=red!100, ultra thick, opacity=0.4] (0,0.8) arc (90:270:0.8);
            \draw[color=blue!100, thick, -] (0,0.8) arc(90:270:0.55 and 0.65);
            \draw[dashed, color=blue!100, thick, -] (0,-0.5) arc(270:450:0.55 and 0.65);
            \draw[-] (-2.5,0) -- (2.5,0) node[anchor=west] {$v=\cosh^{-1}(\omega/\lambda)$};
            \draw[-] (-2.5,1) -- (2.5,1) node[anchor=west] {$v=(1+\epsilon)\cosh^{-1}(\omega/\lambda)$};
            \draw[-] (-2.5,-1) -- (2.5,-1) node[anchor=west] {$v=(1-\epsilon)\cosh^{-1}(\omega/\lambda)$};
            \draw[-] (-2,-1.5) node[anchor=north] {$w=0$} -- (-2,1.5) ;
            \draw[-] (0,-1.5) -- (0,-3/2+1/32) node[anchor=north] {$w=\pi/2$} -- (0,1.5) ;
            \draw[-] (2,-1.5) -- (2,1.5) ;
            \draw[-] (2,-3/2-3/32) node[anchor=north] {$w=\pi$} ;
            \draw[color=blue,thick] (0,0.8) node{$\bullet$} node[anchor=west]{$(w(0),v(0))$};
            \draw[color=cyan,thick] (-1.5,1.25) node[anchor=west]{$\dot{w}=0$}; 
            \draw[cyan, line width = 0.4mm]  plot[smooth,domain=-1.7:0] (\x,{arccosh(2/sin(45*(\x+2)))-arccosh(2)});
        \end{tikzpicture} 
        \caption{The trajectory (blue) with initial condition $(w(0),v(0))$ will hit the curve $\dot{w}=0$ (cyan), and then hit the vertical line $w=\pi/2$. Then by a symmetry argument, the trajectory will eventually return to $(w(0),v(0))$ and form a closed loop around the equilibrium point $(\pi/2,\cosh^{-1}(\omega/\lambda))$.} 
        \label{fig:M2} 
    \end{figure}

Applying the $\delta/n$ criterion Lemma \ref{delta n criterion}, we know that after crossing $\omega-\lambda\sin(w)\cosh(v)=0$ in Fig.~\ref{fig:M2}, the trajectory $(w(t),v(t))$ will stay at the region $\tilde{Q}_{20}$ defined in \eqref{Q tilde} until it hits the region 
\begin{align}  \label{70 line}
        \left\lbrace (w,v) ~\bigg\vert~0<w<\dfrac{\pi}{2},v=\alpha \right\rbrace.
\end{align}

When it enters this region, we notice that the trajectory then enters the third quadrant $Q_{30}$, since $\dot{v}<0$ in set \eqref{70 line}. In $Q_{30}$, $(w(t), v(t))$ cannot go back to the left boundary $w=0$ or swing to the upper boundary $v=\alpha$, since $\dot{w}>0$ and $\dot{v}<0$. Hence the trajectory either first enter $\{0<w\leqslant \pi/2,v=0\}$ or $\{w=\pi/2,0<v\leqslant\alpha\}$. However, using a symmetry argument from Lemma \ref{sym prop} around $w=\pi/2$, the trajectory cannot first touch $\{0<w\leqslant \pi/2,v=0\}$. This follows from uniqueness ($v\equiv 0$ is the solution) and the weak coupling case of real Kuramoto. In other words, 
\begin{align}
    \lim_{t\rightarrow \infty} w(t)=\infty,
\end{align}
if $w(t)$ satisfies the following equation:
\begin{align}
\dot{w}= \;\omega-\lambda\sin(w).
\end{align}
Therefore, the trajectory must pass through $\{w=\pi/2,0<v\leqslant\alpha\}$, but not at the fixed point $(\pi/2,\alpha)$ since $\dot{v}<0$ in the third quadrant $Q_{30}$. 

To recapitulate, we have shown that there exists a time $\tilde{T}>0$ such that $w(\tilde{T})=\frac{\pi}{2}$ and $0<v(\tilde{T})<\alpha$. Incorporating the initial condition \eqref{eq:init_cond} and using Lemma \ref{sym prop}, we prove that the trajectory $(w(t),v(t))$ is symmetric to the vertical line $w=\frac{\pi}{2}$, and hence form a closed contour. Since the system \eqref{n=2} is autonomous, the solution must be periodic with (smallest) period $2\tilde{T}$.

To sum up, every solution of the system of equations \eqref{n=2} with initial conditions \eqref{eq:init_cond} is a periodic orbit. To complete the proof of Theorem \ref{weak N=2}, we still have to check whether the equilibrium point $(\pi/2,\alpha)$ is Lyapunov stable but not asymptotically stable.

Recall Lemma \ref{DA region} and choose $\epsilon>0$. When the initial condition of the system of equations \eqref{n=2} $(w(0),v(0))$, satisfies \begin{align} 
         w(0) = \frac{\pi}{2}
         ,~
         \alpha < v(0) < (1+\epsilon) \alpha,
    \end{align} we have the following 
    \begin{align*}
        \dot{w}(0) &= \omega - \lambda \cosh(v(0))
        < \omega - \lambda \cosh(\alpha) = 0
        ;\\
        \dot{v}(0) &= 0
        ,
    \end{align*}
    so the trajectory $(w(t),v(t))$ will enter $Q_{20}\subset S^-_0$ as soon as $0<t\ll 1$. 
    
    Then according to the Lyapunov analysis carried out above, we have $\mathrm{d}L/\mathrm{d}t<0$ on $S^-_0$, so the distance of $(w(t),v(t))$ to the equilibrium point $(\pi/2,\alpha)$ will strictly decrease as time progresses (until the trajectory leaves $S^-_0$). Equivalently, in the time interval during which the trajectory stays within $S^-_0$, the Lyapunov analysis ensures that the trajectory will be confined in the half disk $\{(w,v)\in S^-_0:~(w-\pi/2)^2+(v-\alpha)^2 \leqslant (v(0)-\alpha)^2 \}$, shown in Fig.~\ref{fig:M2} as the dark red region. 

    By Lemma \ref{DA region}, the deceleration region (green area in Fig. \ref{trajectory}) to the left of the fixed point implies that $\alpha-v(\tilde{T})<v(0)-\alpha$, which holds for every $(w(t),v(t))$ and satisfies the system with $v(0)>\alpha$. By taking a strictly decreasing sequence of $v^{(n)}_0\rightarrow \alpha^+$, one can show that the closed contours, which must be distinct by uniqueness, form a set of nested contours that are contained one within another and are dense inside the periodic orbit containing $(w(0)=\pi/2,v(0))$, resembling the contour produced by a tree's growth rings.
    
    To further clarify this argument, we can first choose another $v^{(1)}_0\in(\alpha,v(0))$ as the starting point. By the above argument, it will produce another periodic orbit.  Denote $u_0^{(1)}\in(0,\alpha)$ as the other $v$-coordinate such that this periodic orbit intersects with $w=\pi/2$ (see Fig.~\ref{fig:M3}). Now, choose any point between these two orbits. Because of the velocity field, $(\dot{w}, \dot{v})$, a trajectory starting from this point will also be between these two orbits, without intersecting them. Therefore, the trajectory must touch $w=\pi/2$ from the right at some finite time and will form a closed loop using Lemma \ref{sym prop}.
    
    If we were to continue in this manner \emph{ad infinitum}, that is, using a sequence of $v^{(n)}_0\in(\alpha,v^{(n-1)}_0)$ for $n=1,2,3,\cdots$ with $v^{(n)}_0 \searrow \alpha$, which generates correspondingly a sequence of $u^{(n)}_0\in(u^{(n-1)}_0,\alpha)$, it can be shown that except the equilibrium point, any point inside the original periodic orbit lies on another periodic orbit. In particular, the ``deceleration region" guaranteed by Lemma \ref{DA region} implies $0<\alpha-u^{(n)}_0<v^{(n)}_0-\alpha$, and thus $u^{(n)}_0 \nearrow \alpha$. This ensures that the open region $U^{(n)}$ enclosed by the periodic orbit passing through $(\pi/2,v^{(n)}_0)$ and $(\pi/2,u^{(n)}_0)$ satisfies $U^{(n)} \supset U^{(n+1)}$ and $\lim\limits_{n\to\infty} \mu(U^{(n)})=0$ where $\mu(\cdot)$ stands for Lebesgue measure. 

    This guarantees that the equilibrium point $(\pi/2,\alpha)$ is Lyapunov stable but not asymptotically stable. The same above argument may be applied for other equilibrium points $(w_k,v_k)$ by exploiting the periodicity. This completes the proof.
\end{proof}

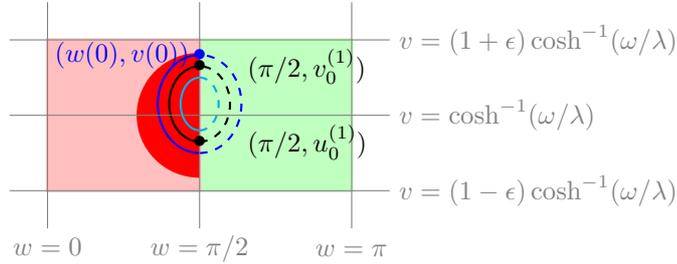
\begin{figure}[H]
\label{trajectory2}
        \centering 
        \begin{tikzpicture}
            \filldraw[color=green!25, fill=green!25, very thick] (0,-1) rectangle (2,1);
            \filldraw[color=red!25, fill=red!25, very thick] (-2,-1) rectangle (0,1);
            \filldraw[color=red!100, ultra thick, opacity=0.4] (0,0.8) arc (90:270:0.8);
            \draw[color=blue!100, thick, -] (0,0.8) arc(90:270:0.55 and 0.65);
            \draw[dashed, color=blue!100, thick, -] (0,-0.5) arc(270:450:0.55 and 0.65);
            \draw[color=black!100, thick, -] (0,0.65) arc(90:270:0.4 and 0.5);
            \draw[dashed, color=black!100, thick, -] (0,-0.35) arc(270:450:0.4 and 0.5);
            \draw[color=cyan!100, thick, -] (0,0.5) arc(90:270:0.25 and 0.35);
            \draw[dashed, color=cyan!100, thick, -] (0,-0.2) arc(270:450:0.25 and 0.35);
            \draw[-,color=gray] (-2.5,0) -- (2.5,0) node[anchor=west] {$v=\cosh^{-1}(\omega/\lambda)$};
            \draw[-,color=gray] (-2.5,1) -- (2.5,1) node[anchor=west] {$v=(1+\epsilon)\cosh^{-1}(\omega/\lambda)$};
            \draw[-,color=gray] (-2.5,-1) -- (2.5,-1) node[anchor=west] {$v=(1-\epsilon)\cosh^{-1}(\omega/\lambda)$};
            \draw[-,color=gray] (-2,-1.5) node[anchor=north] {$w=0$} -- (-2,1.5) ;
            \draw[-,color=gray] (0,-1.5) -- (0,-3/2+1/32) node[anchor=north] {$w=\pi/2$} -- (0,1.5) ;
            \draw[-,color=gray] (2,-1.5) -- (2,1.5) ;
            \draw[-,color=gray] (2,-3/2-3/32) node[anchor=north] {$w=\pi$} ;
            \draw[color=blue,thick] (0,0.8) node{$\bullet$} node[anchor=east]{$(w(0),v(0))$};
            \draw[color=black,thick] (0,0.65) node{$\bullet$};
            \draw[color=black,thick] (0.5,0.65) node[anchor=west]{$(\pi/2,v^{(1)}_0)$};
            \draw[color=black,thick] (0,-0.35) node{$\bullet$};
            \draw[color=black,thick] (0.5,-0.35) node[anchor=west]{$(\pi/2,u^{(1)}_0)$};
        \end{tikzpicture} 
        \caption{Inside the blue trajectory with initial condition $(w(0),v(0))$, the black trajectory with initial condition $(\pi/2,v_0^{(1)})$ also forms a closed loop via the same argument. Denote $(\pi/2,u_0^{(1)})$ as the other intersection point of the black trajectory with the line $w=\pi/2$. Similarly, the cyan trajectory with initial condition $(\pi/2,v_0^{(2)})$ forms a closed loop inside the black trajectory, and so on \emph{ad infinitum}.} 
        \label{fig:M3} 
\end{figure}

\begin{figure}[H] 
\centering
\includegraphics[width=12cm]{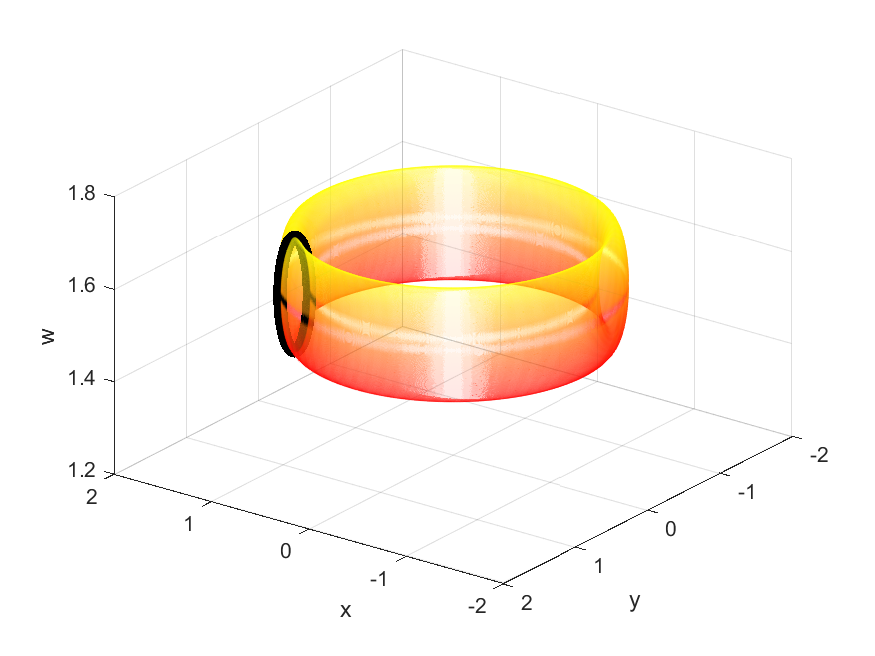}
\caption{The ``peach ring" surface is the collection of flows for the system equations \eqref{n=2}: $\{ \phi^t(w(0),x(0),y(0),z(0)): w(0)=\pi/2, |x(0)|^2+|y(0)|^2=2.1025, z(0)=0\}$ with the parameters $\omega = 2$ and $\lambda = 1$. The black ring lying on the surface is the solution of the system equations \eqref{n=2} with the same parameters and initial condition: $w(0)=\pi/2$, $x(0)=1$, $y(0)=1.05$, $z(0)=0$.}
\label{peachring}
\end{figure}

\begin{subsection}{Three oscillators} \label{sec 4.2}
    In this subsection, we demonstrate the case $N=3$. In particular, we analyze the so-called ``Lion Dance flow," that is, the case when assuming the following two conditions in \eqref{main eq} and \eqref{def of v_mn}:
    \begin{enumerate}
        \item $\omega_1-\omega_2=\omega_2-\omega_3>0$.
        \item $q_1(0)-q_2(0)=q_2(0)-q_3(0)$. That is,
            \begin{itemize}
                \item $w_1(0)-w_2(0)=w_2(0)-w_3(0)$.
                \item $x_1(0)-x_2(0)=x_2(0)-x_3(0)$.
                \item $y_1(0)-y_2(0)=y_2(0)-y_3(0)$.
                \item $z_1(0)-z_2(0)=z_2(0)-z_3(0)$.
            \end{itemize} 
    \end{enumerate}
    Notice that the set of possible solutions satisfying these assumptions constitutes an invariant manifold based on uniqueness. Because of these assumptions, we have that for all $t\geqslant 0$, $q_1(t)-q_2(t)=q_2(t)-q_3(t)$. Thus
    \begin{align*}
        v_{12}(t)=v_{23}(t);
        \quad
        v_{13}(t) = 2 v_{12}(t).
    \end{align*}
    Denoting 
    \begin{align*}
        \omega
        :=
        \omega_1 - \omega_3
        =
        2 (\omega_1 - \omega_2)
        ;\quad
        w(t)
        :=
        w_1(t) - w_2(t)
        ;\quad
        v(t)
        :=
        v_{12}(t)
    \end{align*}
     for brevity, the simplified system equations  \eqref{main eq} will simplify to the following
     \begin{equation} \label{n=3}
\left\{ \begin{aligned}
    \dot{w}=& \;
    \dfrac{\omega}{2}
    -
    \dfrac{\lambda}{3}(\sin(w)\cosh(v)+\sin(2w)\cosh(2v)),\\
    \dot{v}=&~~~~-\dfrac{\lambda}{3}(\cos(w)\sinh(v)+\cos(2w)\sinh(2v)).
\end{aligned}
\right.
\end{equation}
Notice that in this case, $\omega=\lambda_c$ where $\lambda_c$ is defined in \eqref{eq:lambda_c}. Now, we define three regimes of weak coupling:
\begin{enumerate}
    \item Super weak: $ \lambda < \Lambda_c:=3\omega/(2\max\limits_{x\in \mathbb{R}}(\sin(x)+\sin(2x)))
    \approx 0.85218915 \omega $. 
    \item Critically weak: $\lambda=\Lambda_c$.
    \item Weak: $ \Lambda_c < \lambda < \lambda_c = \omega $.
\end{enumerate}

\subsubsection{``Super weak" case $\lambda<\Lambda_c$:} \label{4.2.1}
It is highly nontrivial to observe that in this case the system of equations \eqref{n=3} possess two equilibria in each set:
\begin{align}  \label{75 set}
        R_k:=\left\lbrace (w,v) ~\bigg\vert~2k \pi \leqslant w< 2(k+1) \pi,~v>0\right\rbrace,
\end{align}
where $k\in \mathbb{Z}$. We claim that one of these equilibria is a sink $(w^k,v^k)$ and the other is a source $(\bar{w}^k,\bar{v}^k)$. Due to periodicity, we focus on the equilibrium point $(w^0,v^0)$ and $(\bar{w}^0,\bar{v}^0)$ in the set $R_0$. To analyze the stability of each equilibrium point, we need to locate the roots of the system of equations:
\begin{equation} \label{RHS n=3}
\left\{ \begin{aligned}
    \dfrac{\omega}{2}-\dfrac{\lambda}{3}(\sin(w)\cosh(v)+\sin(2w)\cosh(2v))=0,\\
    -\dfrac{\lambda}{3}(\cos(w)\sinh(v)+\cos(2w)\sinh(2v))=0.
\end{aligned}
\right.
\end{equation}

We observe that $v=0$ cannot be a solution of system \eqref{RHS n=3} when the coupling strength is super weak $\lambda<\Lambda_c$. Using hyperbolic function identities, the second equation in system equations \eqref{RHS n=3} yields
\begin{align} \label{77}
    \cos(w)+2\cos(2w)\cosh(v)=0.
\end{align}
By \eqref{77}, we obtain
\begin{align} \label{78}
    \cosh(v)=-\dfrac{\cos(w)}{2\cos(2w)},
\end{align}
which is well-defined, since $\cos(2w)$ cannot be zero.

Recalling the hyperbolic function identities for $\cosh(2v)$ and substituting equation system equations \eqref{78} into first equation in \eqref{RHS n=3}, we arrive at
\begin{align} 
    \dfrac{3\omega}{2\lambda}=\sin(w)\left(-\dfrac{\cos(w)}{2\cos(2w)}+2\cos(w)\left(\dfrac{\cos^2(w)}{2\cos^2(2w)}
    -1\right)\right).
\end{align}
We exploit the trigonometric identities and make a straightforward calculation to obtain
\begin{align} \label{80}
    4\sin^3(2w)+\dfrac{6\omega}{\lambda}\sin^2(2w)-3\sin(2w)-\dfrac{6\omega}{\lambda}=0.
\end{align}
Let the cubic polynomial having a root $\sin(2w)$ be
\begin{align}
    p(x)=4x^3+\dfrac{6\omega}{\lambda}x^2-3x-\dfrac{6\omega}{\lambda}.
\end{align}
We observe that $p(1)=1>0$ and $p(0)=-6\omega/\lambda<0$. This shows that there exists at least one root of $p(x)$ in $(0,1)$. Also, $p'(0)=-3<0$  and $p'(-\infty)>0$, so this implies that we cannot have three roots of $p(x)$ in $(0,1)$. 

Now we find that there exists $\tilde{w}_l\in (0,2\pi)$, $l=1,2,3,4$ such that $\tilde{w}_1<\tilde{w}_2<\tilde{w}_3<\tilde{w}_4$ and $\sin(2 \tilde{w}_l)$ satisfies equation \eqref{80}. However, $\tilde{w}_2$ and $\tilde{w}_4$ do not satisfy the system equations \eqref{RHS n=3}. In order to explain this, we make the following observations.\\

\begin{lemma} [Horizontal cutting lemma] \label{lemma 4.5}
    Consider a continuous function $f:\mathbb{R}^2\rightarrow \mathbb{R}$. Let $g_1: [y_1,y_2]\rightarrow \mathbb{R}$ and $g_2: [x_1,x_2]\rightarrow \mathbb{R}$ be two strictly monotone continuous functions such that $g_1(y_1)=g_2(x_1)$ and $g_1(y_2)=g_2(x_2)$. If $f(x_1,y_1)\cdot f(x_2,y_2)<0$, then the system of equations:
\begin{equation} \label{horizontal cutting lemma}
\left\{ \begin{aligned}
    f(x,y)&=0,\\
    g_2(x)-g_1(y)&=0,
\end{aligned}
\right.
\end{equation}
has at least one root $(\tilde{x},\tilde{y})$, and hence $x_1 < \tilde{x} < x_2$ and $y_1 < \tilde{y} < y_2$.
\end{lemma}

\begin{proof}
Let the continuous function $G$ defined by $G(x)=g^{-1}_1 g_2(x)$. Clearly, because of intermediate value theorem, $G: [x_1,x_2]\rightarrow [y_1,y_2]$ is well-defined. Therefore, we obtain $f(x_1,G(x_1))\cdot f(x_2,G(x_2))<0$. Using the intermediate value theorem again, we complete the proof.
\end{proof}

Let
\begin{align} \label{f function}
    f(w,v)
    :=
    \dfrac{3\omega}{2\lambda}
    -(\sin(w)\cosh(v)+\sin(2w)\cosh(2v))
\end{align}
be defined on $(w,v)\in\mathbb{R}^2,$ 
\begin{align} \label{g1}
    g_1(v):=2\cosh(v)
\end{align} 
be defined on $\mathbb{R}^+$ and
\begin{align} \label{g2}
    g_2(w):=-\dfrac{\cos(w)}{\cos(2w)}
\end{align}
be defined on the domain 
\begin{align}
    \{w\in[0,2\pi):-\dfrac{\cos(w)}{\cos(2w)}\geqslant 2\}
    =
    \left(\dfrac{\pi}{4},r_1\right]
    \bigcup 
    \left(\dfrac{3\pi}{4},r_2\right] \bigcup 
    \left[r_3,\dfrac{5\pi}{4}\right) \bigcup 
    \left[r_4,\dfrac{7\pi}{4}\right),
\end{align}
where $r_l$, $l=1,2,3,4$ satisfies $g_2(r_l)=g_1(0)=2$ and $\pi/4<r_1<3\pi/4<r_2<\pi<r_3<5\pi/4<r_4<7\pi/4$ (see Fig.~\ref{fig:M4}).
\begin{figure} 
        \centering 
        \begin{tikzpicture}
            \draw[step=1cm,gray,very thin] (0,-4) grid (6.28,4);
            \draw[black] (0,1/16) -- (0,0) node{$\bullet$} node[anchor=north]{$0$};
            \draw[black] (3.14,1/16) --  (3.14,0) node{$\bullet$} node[anchor=north]{$\pi$};
            \draw[black] (6.28,1/16) --  (6.28,0) node{$\bullet$} node[anchor=north]{$2\pi$};
            \draw[->] (0,0) -- (6.78,0) node[anchor=west] {$w$};
            \draw[->] (0,-4) -- (0,4);
            \draw[red, line width = 0.4mm]  plot[smooth,domain=0:0.68775] (\x,{-cos(deg(\x))/cos(deg(2*\x)) });
            \draw[dashed, red, line width = 0.2mm] (pi/4,-4) node[anchor=north]{$\frac{\pi}{4}$} -- (pi/4,4) ;
            \draw[red, line width = 0.4mm]  plot[smooth,domain=0.8675:2.275] (\x,{-cos(deg(\x))/cos(deg(2*\x)) });
            \draw[dashed, red, line width = 0.2mm] (3*pi/4,-4) node[anchor=north]{$\frac{3\pi}{4}$} -- (3*pi/4,4) ;
            \draw[red, line width = 0.4mm]  plot[smooth,domain=2.4545:3.83] (\x,{-cos(deg(\x))/cos(deg(2*\x)) });
            \draw[dashed, red, line width = 0.2mm] (5*pi/4,-4) node[anchor=north]{$\frac{5\pi}{4}$} -- (5*pi/4,4) ;
            \draw[red, line width = 0.4mm]
            plot[smooth,domain=4.0082:5.4165] (\x,{-cos(deg(\x))/cos(deg(2*\x)) });
            \draw[dashed, red, line width = 0.2mm] (7*pi/4,-4) node[anchor=north]{$\frac{7\pi}{4}$} -- (7*pi/4,4) ;
            \draw[red, line width = 0.4mm]
            plot[smooth,domain=5.595:6.28] (\x,{-cos(deg(\x))/cos(deg(2*\x)) });
            \draw[red] (3.14,4) node[anchor=south] {$g_2(w):=-\dfrac{\cos(w)}{\cos(2w)}$};
            \draw[step=1cm,gray,very thin] (8,-4) grid (14,4);
            \draw[->] (8,0) -- (14,0) node[anchor=west] {$v$};
            \draw[->] (11,-4) -- (11,4);
            \draw[black] (11,1/16) -- (11,0) node{$\bullet$} node[anchor=north]{$0$};
            \draw[cyan, line width = 0.4mm]
            plot[smooth,domain=9.69:12.32] (\x,{2*cosh(\x-11)});
            \draw[black,-] (11,2);
            \draw[cyan] (11,4.25) node[anchor=south] {$g_1(v):=2\cosh(v)$};
            \draw[blue, line width = 0.4mm] (0,2) node[anchor=east] {$2$} -- (14,2) node[anchor=west] {$2$};
            \draw[blue,dashed] (0.9359,2) node{$\bullet$} -- (0.9359,0) node{$\bullet$} node[anchor=north]{$r_1$} ;
            \draw[blue,dashed] (2.5737,2) node{$\bullet$} -- (2.5737,0) node{$\bullet$} node[anchor=north]{$r_2$} ;
            \draw[blue,dashed] (3.7094,2) node{$\bullet$} -- (3.7094,0) node{$\bullet$} node[anchor=north]{$r_3$} ;
            \draw[blue,dashed] (5.3473,2) node{$\bullet$} -- (5.3473,0) node{$\bullet$} node[anchor=north]{$r_4$} ;
            \draw[black] (11,2+1/16) -- (11,2) node{$\bullet$};
        \end{tikzpicture} 
        \caption{These two plots are illustrations of Lemma~\ref{lemma 4.5} for the specific functions $g_1(v)$ (right) and $g_2(w)$ (left). The indigo line shows that for $l=1,2,3,4,$ $g_2(r_l)=g_1(0)=2$.} 
        \label{fig:M4} 
    \end{figure}
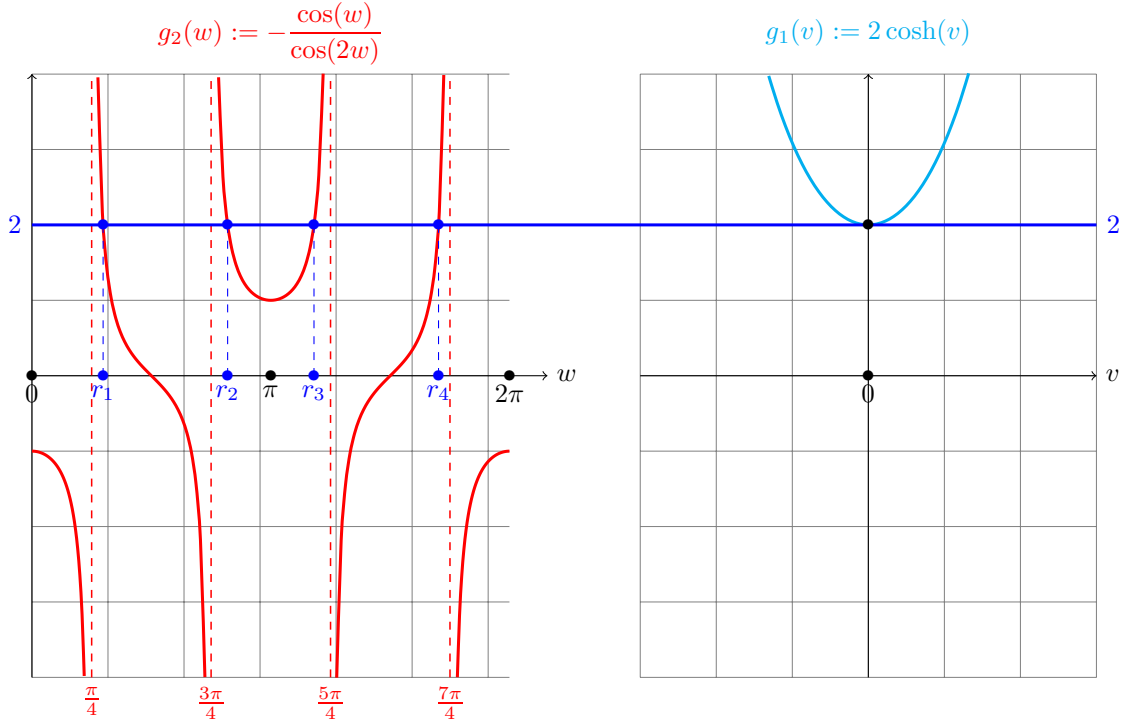
A straightforward calculation reveals that 
\begin{align}
    -\dfrac{7}{8}<\cos(r_l)=\dfrac{-1-\sqrt{33}}{8}<-\dfrac{3}{4}, ~\mbox{for}~l=2,3,
\end{align}
and
\begin{align}
    \dfrac{1}{2}<\cos(r_l)=\frac{-1+\sqrt{33}}{8}<\dfrac{5}{8}, ~\mbox{for}~l=1,4.
\end{align}
Clearly, $\sin(2w)<0$, when $w\in (\frac{3\pi}{4},r_2]\bigcup [r_4,\frac{7\pi}{4})$, which is a contradiction to $\sin(2w)\in (0,1)$ in \eqref{80}. By a direct calculation, we have
\begin{align} \label{87}
    f(r_1,0)&=\dfrac{3\omega}{2\lambda}-\sin(r_1)-\sin(2r_1)
    =
    \dfrac{3\omega}{2\lambda}
    - \max\limits_{x\in\mathbb{R}} (\sin(x)+\sin(2x))
    ,\\ \label{88}
    f\left(\left(\dfrac{\pi}{4}\right)^+,\infty\right)&=-\infty.
\end{align}
The second equality in \eqref{87} can be seen via Fig.~\ref{fig:M5}, noting that $g_2(w)=2$ if and only if $h'(w)=0$, where $h(w):=\sin(w)+\sin(2w)$. In the ``super weak" regime, from \eqref{87} we have $f(r_1,0)>0$. Hence by means of \eqref{87}, \eqref{88} and Lemma \ref{lemma 4.5}, we prove that 
there exits $\tilde{w}_1\in(\pi/4,r_1)$, $\tilde{v}_1>0$ such that $(\tilde{w}_1,\tilde{v}_1)$ is a root for system of equations \eqref{RHS n=3}.

Similarly, we have
\begin{align} \label{89}
    f(r_3,0)&=\dfrac{3\omega}{2\lambda}-\sin(r_3)-\sin(2r_3)
    \approx
    \dfrac{3\omega}{2\lambda}
    - 0.369
    > 0 ,\\ \label{90}
    f\left(\left(\dfrac{5\pi}{4}\right)^-,\infty\right)&=-\infty.
\end{align}
Notice that via \eqref{89}, in both the ``super weak" and the ``weak" regimes, we have $f(r_3,0)>0$. Combining \eqref{89}, \eqref{90} and Lemma \ref{lemma 4.5}, it is immediate that there exists $\tilde{w}_3\in(r_3,5\pi/4)$, $\tilde{v}_3>0$ such that $(\tilde{w}_3,\tilde{v}_3)$ is a root for system of system equations \eqref{RHS n=3}. To recapitulate, we proved that system equations \eqref{n=3} only have two equilibrium points in $R_0$, and hence one of them is in $(\pi/4,r_1)$ and the other is in $(r_3,5\pi/4)$.

Linearizing system equations \eqref{n=3} about $(\tilde{w}_l,\tilde{v}_l)$, $l=1,3$, we arrive at
\begin{align}
\begin{pmatrix} \dot{w} \\ \dot{v}  \end{pmatrix}
=\begin{pmatrix} 
L_{11}(\tilde{w}_l,\tilde{v}_l) ~~L_{12}(\tilde{w}_l,\tilde{v}_l)\\ L_{21}(\tilde{w}_l,\tilde{v}_l)~~L_{22}(\tilde{w}_l,\tilde{v}_l)
\end{pmatrix}
\begin{pmatrix} w \\ v  \end{pmatrix}:=L(\tilde{w}_l,\tilde{v}_l)\begin{pmatrix} w \\ v  \end{pmatrix},
\end{align}
where
\begin{align}
    L_{11}(\tilde{w}_l,\tilde{v}_l)=L_{22}(\tilde{w}_l,\tilde{v}_l)&=-\dfrac{\lambda}{3}\left( \cos(\tilde{w}_l)\cosh(\tilde{v}_l)+2\cos(2\tilde{w}_l)\cosh(2\tilde{v}_l)\right),\\
    -L_{21}(\tilde{w}_l,\tilde{v}_l)=L_{12}(\tilde{w}_l,\tilde{v}_l)&=-\dfrac{\lambda}{3}\left( \sin(\tilde{w}_l)\sinh(\tilde{v}_l)+2\sin(2\tilde{w}_l)\sinh(2\tilde{v}_l)\right).
\end{align}
A straightforward calculation reveals that the eigenvalues $\lambda_{l1}$ and $\lambda_{l2}$ of the matrix $L(\tilde{w}_l,\tilde{v}_l)$ have same real part: 
\begin{align} \label{real part of eigenvalue}
    \mathfrak{Re}(\lambda_{l1})=\mathfrak{Re}(\lambda_{l2})=-\dfrac{\lambda}{3}\left( \cos(\tilde{w}_l)\cosh(\tilde{v}_l)+2\cos(2\tilde{w}_l)\cosh(2\tilde{v}_l)\right).
\end{align}
We observe that by using the second equation of system equations \eqref{RHS n=3}, the real part of the eigenvalues may be written as
\begin{align} \label{eq:98}
   \mathfrak{Re}(\lambda_{l1})=\mathfrak{Re}(\lambda_{l2})
   =
   -
   \frac{\lambda}{3}\cos(\tilde{w}_l)\left( \cosh(\tilde{v}_l)-\dfrac{2\sinh(\tilde{v}_l)\cosh(2\tilde{v}_l)}{\sinh(2\tilde{v}_l)}\right).
\end{align}
Since for all $v>0$, 
\begin{align*} 
\cosh(v)-\dfrac{2\sinh(v)\cosh(2v)}{\sinh(2v)}
&=
\cosh(v) - \dfrac{\cosh(2v)}{\cosh(v)}
=
\dfrac{\cosh^2(v)-\cosh(2v)}{\cosh(v)}
\\
&=
\dfrac{\cosh^2(v)-(2\cosh^2(v)-1)}{\cosh(v)}
=
\dfrac{1-\cosh^2(v)}{\cosh(v)}
<0
,    
\end{align*}
it is straightforward to see that $\mathfrak{Re}(\lambda_{l1})=\mathfrak{Re}(\lambda_{l2})>0$ when $l=1$ and, on the other hand, $\mathfrak{Re}(\lambda_{l1})=\mathfrak{Re}(\lambda_{l2})<0$ when $l=3$. Because of Hartman–Grobman theorem, we verify our claim. To recapitulate, we proved that in $R_0$, one of these equilibria is a sink $(\tilde{w}_3,\tilde{v}_3):=(w^0,v^0)$ and the other is a source $(\tilde{w}_1,\tilde{v}_1):=(\bar{w}^0,\bar{v}^0)$. Finally, the periodicity let us conclude that, in $R_k$, there exists only one sink $(w^k,v^k)$ such that $w^k\in(2k\pi+r_3,2k\pi+5\pi/4)$ and only one source $(\bar{w}^k,\bar{v}^k)$ such that $\bar{w}^k\in(2k\pi+\pi/4,2k\pi+r_1)$.

\subsubsection{``Critically weak" and ``weak" cases $\Lambda_c \leqslant \lambda<\lambda_c=\omega$:} \label{4.2.2}

First, we focus on the ``critically weak" case $\lambda=\Lambda_c$. We claim that the system of equations \eqref{n=3} possesses two equilibria. Thus we may verify that one of the equilibrium points is $(w,v)=(r_1,0)$; see Fig.~\ref{fig:M5}. This equilibrium is semistable, since the stable manifold $W^{s}(r_1,0)$ contains at least $\{ (w,v): 0<w \leqslant r_1,~v=0\}$ and on the other hand, the unstable manifold $W^{u}(r_1,0)$ contains at least $\{ (w,v): r_1<w<2\pi,~v=0\}$, which follows from the real Kuramoto model. 

Following a similar argument from the super weak coupling and recalling the definition of $f$, $g_1$ and $g_2$ in \eqref{f function},\eqref{g1} and \eqref{g2}, we notice that any equilibrium point with $v>0$ cannot exist in either $w\in (\frac{3\pi}{4},r_2) \cup (r_4,\frac{7\pi}{4})$ or $\left(\frac{\pi}{4},r_1\right)$, since $f(w,g^{-1}_1 g_2(w))<0$ when $w\in \left(\frac{\pi}{4},r_1\right)$. 

This latter fact can be shown via a straightforward argument using calculus. First, we observe that $\frac{\mathrm{d}^2}{\mathrm{d}w^2} f(w,g^{-1}_1 g_2(w)) < 0$ when $w\in \left(\frac{\pi}{4},r_1\right)$ once the second derivative is calculated and simplified. This implies that the derivative $\frac{\mathrm{d}}{\mathrm{d}w} f(w,g^{-1}_1 g_2(w)) $ is strictly decreasing. Since it can be readily computed that $\frac{\mathrm{d}}{\mathrm{d}w} f(w,g^{-1}_1 g_2(w))\vert_{w=r_1}>0$, this implies that $\frac{\mathrm{d}}{\mathrm{d}w} f(w,g^{-1}_1 g_2(w))>0$ on  $\left(\frac{\pi}{4},r_1\right)$. Then, observing that $f(w,g^{-1}_1 g_2(w))\vert_{w=r_1}=f(r_1,0)=\frac{3\omega}{2\lambda}-(\sin(r_1)+\sin(2r_1))$ from \eqref{87} is not positive in the (critically) weak regime $\Lambda_c \leqslant \lambda<\lambda_c$, we conclude that $f(w,g^{-1}_1 g_2(w)) < 0$ when $w\in \left(\frac{\pi}{4},r_1\right)$. 

Using Lemma \ref{lemma 4.5} and applying \eqref{89} and \eqref{90}, we obtain that there is a second equilibrium point belonging in $\left(r_3,\frac{5\pi}{4}\right)$. Since this equilibrium point is in $\left(\pi,\frac{5\pi}{4}\right)$, on which cosine is negative, from \eqref{eq:98} we conclude that the equilibrium point is asymptotically stable (i.e., a sink). 

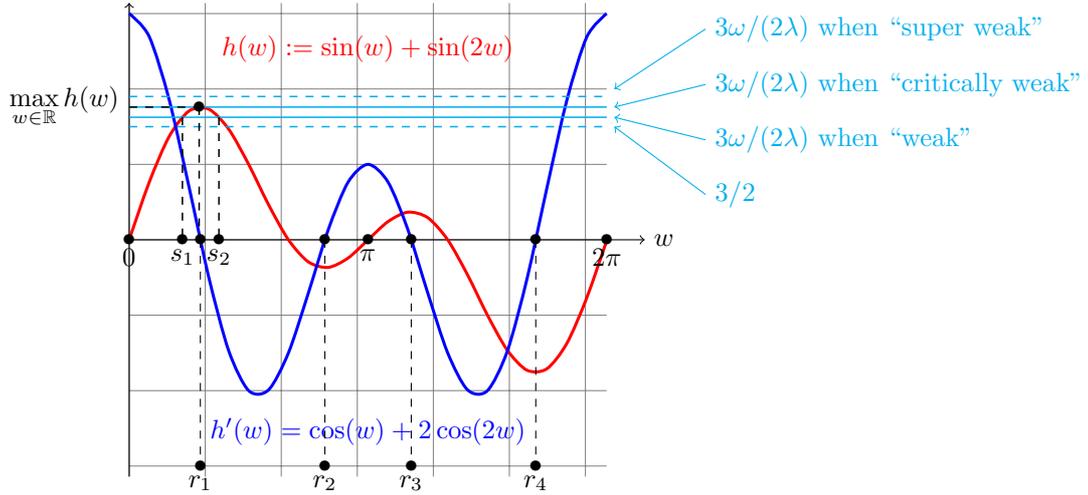
\begin{figure} \label{fig:M5}
        \centering 
        \begin{tikzpicture}
            \draw[step=1cm,gray,very thin] (0,-3.14) grid (6.28,3.14);
            \draw[->] (0,0) -- (6.78,0) node[anchor=west] {$w$};
            \draw[->] (0,-3.14) -- (0,3.14);
            \draw[red, line width = 0.4mm]  plot[smooth,domain=0:6.28] (\x,{sin(deg(\x))+sin(deg(2*\x))});
            \draw[red] (3.14,2.24) node[anchor=south] {$h(w):=\sin(w)+\sin(2w)$};
            \draw[blue, line width = 0.4mm]  plot[smooth,domain=-0:6.28] (\x,{cos(deg(\x))+2*cos(deg(2*\x))});
            \draw[blue] (3.14,-2.24) node[anchor=north] {$h'(w)=\cos(w)+2\cos(2w)$};
            \draw[cyan, line width = 0.2mm]  (0,1.625) -- (6.28,1.625) ;
            \draw[cyan,->] (7.58,1.325) node[anchor=west] {$3\omega/(2\lambda)$ when ``weak"} -- (6.38,1.625) ;
            \draw[cyan, line width = 0.2mm]  (0,1.76) -- (6.28,1.76) ;
            \draw[cyan,->] (7.58,2.06) node[anchor=west] {$3\omega/(2\lambda)$ when ``critically weak"} -- (6.38,1.76) ;
            \draw[dashed, cyan, line width = 0.2mm]  (0,1.9) -- (6.28,1.9) ;
            \draw[cyan,->] (7.58,2.795) node[anchor=west] {$3\omega/(2\lambda)$ when ``super weak"} -- (6.38,2) ;
            \draw[dashed, cyan, line width = 0.2mm] (0,1.5) -- (6.28,1.5) ;
            \draw[cyan,->] (7.58,0.6) node[anchor=west] {$3/2$} -- (6.38,1.5) ;
            \draw[dashed, black, line width = 0.2mm](0,1.76) node[anchor=east]{$\max\limits_{w\in\mathbb{R}} h(w)$} -- (0.92,1.76) node{$\bullet$} ;
            \draw[dashed, black, line width = 0.2mm] (0.92,1.76) -- (0.92,0) ;
            \draw[dashed, black, line width = 0.2mm] (0.70,1.625) -- (0.70,0) node{$\bullet$} node[anchor=north] {$s_1$} ;
            \draw[dashed, black, line width = 0.2mm] (1.18,1.625) -- (1.18,0) node{$\bullet$} node[anchor=north] {$s_2$} ;
            \draw[black] (0,1/16) -- (0,0) node{$\bullet$} node[anchor=north]{$0$};
            \draw[black] (3.14,1/16) --  (3.14,0) node{$\bullet$} node[anchor=north]{$\pi$};
            \draw[black] (6.28,1/16) --  (6.28,0) node{$\bullet$} node[anchor=north]{$2\pi$};
            \draw[black,dashed] (0.9359,0) node{$\bullet$} -- (0.9359,-3) node{$\bullet$} node[anchor=north]{$r_1$} ;
            \draw[black,dashed] (2.5737,0) node{$\bullet$} -- (2.5737,-3) node{$\bullet$} node[anchor=north]{$r_2$} ;
            \draw[black,dashed] (3.7094,0) node{$\bullet$} -- (3.7094,-3) node{$\bullet$} node[anchor=north]{$r_3$} ;
            \draw[black,dashed] (5.3473,0) node{$\bullet$} -- (5.3473,-3) node{$\bullet$} node[anchor=north]{$r_4$} ;
        \end{tikzpicture} 
        \caption{This plot is an illustration of the function $h(w)$ and its derivative $h'(w)$ on the interval $[0,2\pi]$. We also clarify the values of $3\omega/(2\lambda)$ for "super weak", ``critically weak" and ``weak" coupling in cyan: ``super weak" when $3\omega/(2\lambda) > \max\limits_{w\in\mathbb{R}} h(w)$, ``critically weak" when $3\omega/(2\lambda)=\max\limits_{w\in\mathbb{R}} h(w)$, and ``weak" when $3/2<3\omega/(2\lambda)<\max\limits_{w\in\mathbb{R}} h(w)$.} 
        \label{fig:M5} 
    \end{figure}

Finally, we restrict our attention to the ``weak" case. When $\Lambda_c < \lambda<\lambda_c$, the system equations \eqref{n=3} possesses three equilibria. Based on Fig.~\ref{fig:M5}, we observe that two of these equilibria are $(w,v)=(s_1,0)$ and $(w,v)=(s_2,0)$. These two equilibrium points bifurcate from $(r_1,0)$ when $\lambda$ exceeds the critically weak coupling strength $\Lambda_c$. As $h'(s_1)>0$ and $h'(s_2)<0$, we can conclude that $(s_1,0)$ is a sink and $(s_2,0)$ is a source by the formula stated in \eqref{real part of eigenvalue}. Finally, we find that the third equilibrium point is a sink, whose $w$-coordinate belongs to $\left[r_3,\frac{5\pi}{4}\right)$ and $v$-coordinate is strictly positive from \eqref{89}, \eqref{90} and the discussion immediately below the two equations.

\newpage

\begin{figure}[H] 
\centering
\includegraphics[width=0.5\textwidth]{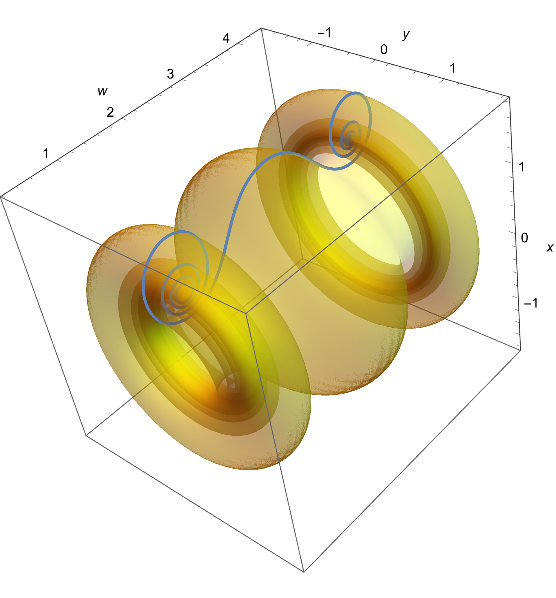}
\caption{This is a 3D plot of the phase-locking synchronization solutions for ``super weak" coupling in $N=3$ with parameters $\omega = 1$ and $\lambda = \Lambda_c/2 \approx 0.426095$. The blue curve is the solution of the system of equations \eqref{n=3} with initial conditions: $w(0)=1$, $x(0)=0.8$, $y(0)=0$, $z(0)=0$. The yellow surface is the collection of flows for the system of equations \eqref{n=3}: $\{ \phi^t(w(0),x(0),y(0),z(0)): w(0)=1, |x(0)|^2+|y(0)|^2=0.64, z(0)=0\}$.}
\end{figure}

\begin{figure}[H] 
\centering
\includegraphics[width=0.5\textwidth]{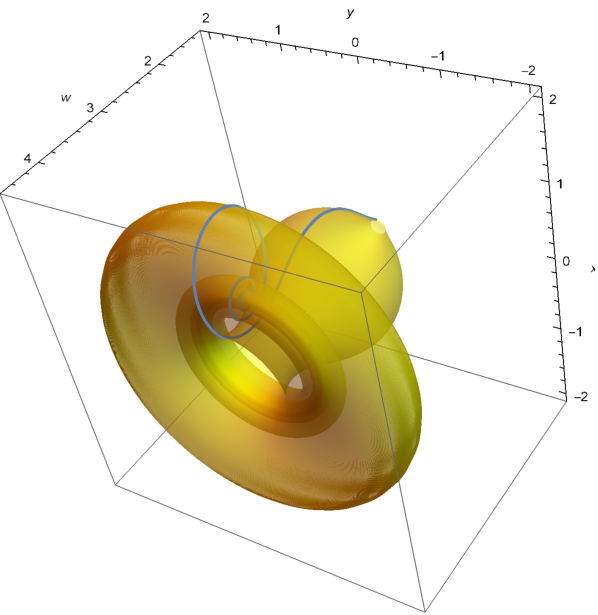}
\caption{This is a 3D plot of the phase-locking synchronization solutions for ``weak" coupling in $N=3$ with parameters $\omega = 1$ and $\lambda = (1 + \Lambda_c)/2 \approx 0.963047$. The blue curve is the solution of the system of equations \eqref{n=3} with initial conditions: $w(0)=1.4$, $x(0)=0.1$, $y(0)=0$, $z(0)=0$. The yellow surface is the collection of flows for the system of equations \eqref{n=3}: $\{ \phi^t(w(0),x(0),y(0),z(0)): w(0)=1.4, |x(0)|^2+|y(0)|^2=0.01, z(0)=0\}$.}
\end{figure}

\newpage

\end{subsection}

\section*{Conclusion} \label{sec 5}
We have defined the Kuramoto model with quaternions and identified the features when the system is under strong and weak coupling strength, $\lambda$. Firstly, a Kuramoto system of $N$ oscillators under strong coupling will be asymptotically stable and will achieve phase-locking and frequency synchronization. Secondly, for the case when coupling strength is weak, we demonstrate that periodic orbits exist for the $N=2$ oscillator system, and its phase-locking state is Lyapunov stable. For the $N=3$ oscillator system, we consider the ``Lion Dance" flow and show that one of its phase-locking states is asymptotically stable and the other is unstable when the coupling strength is super weak. If the coupling strength exceeds critical point $\lambda > \Lambda_c$, then there is an unstable phase-locking state that bifurcates into stable and unstable phase-locking states. 

We also numerically investigated the behavior of the Lion Dance flow when $N\geqslant 4$. Applying the assumptions that $\omega_i-\omega_{i+1}=\omega_{i+1}-\omega_{i+2}>0$ and $q_i(0)-q_{i+1}(0)=q_{i+1}(0)-q_{i+2}(0)$ for $i=1,\ldots,N-2$ into the quaternion Kuramoto model \eqref{main eq} with \eqref{def of v_mn}, we obtain the following simplified equations
\begin{equation} \label{lion}
\left\{ \begin{aligned}
    \dot{w}=& \;
    \dfrac{\omega}{N-1}
    -
    \dfrac{\lambda}{N}\sum_{m=1}^{N-1} \sin(mw)\cosh(mv),\\
    \dot{v}=&~~~~~~~~-\dfrac{\lambda}{N}\sum_{m=1}^{N-1} \cos(mw)\sinh(mv),
\end{aligned}
\right.
\end{equation}
where we denoted $\omega:=\omega_1-\omega_N=\lambda_c$, and the functions $w(t):=w_i(t)-w_{i+1}(t)$, $v(t):=\sqrt{(x_i(t)-x_{i+1}(t))^2+(y_i(t)-y_{i+1}(t))^2+(z_i(t)-z_{i+1}(t))^2}$ for any $i=1,\ldots,N-1$ are well-defined due to uniqueness of solution to \eqref{main eq}.

Based on the numerical computation in Fig.~\ref{fig:100}, we conjecture that when $N \geqslant 4$ the number of asymptotically stable phase-locking states of Lion Dance flow is $\lfloor \frac{N-1}{2}\rfloor$ when the coupling strength is ``super weak" or ``critically weak", i.e., $\lambda \leqslant \Lambda_c$. 

There remain many aspects of the quaternionic Kuramoto model to be explored, such as the stability of Lion Dance flow when $N\geqslant 4$, the general topological structure for both strong and weak coupling strengths, the effect of time delay and the behavior of the system as the number of the oscillators tends to infinity using the mean field approximation.

\begin{figure}[h]
\centering
\subfigure[Lion Dance flow, $N=3$.]{\includegraphics[width=0.46\textwidth]{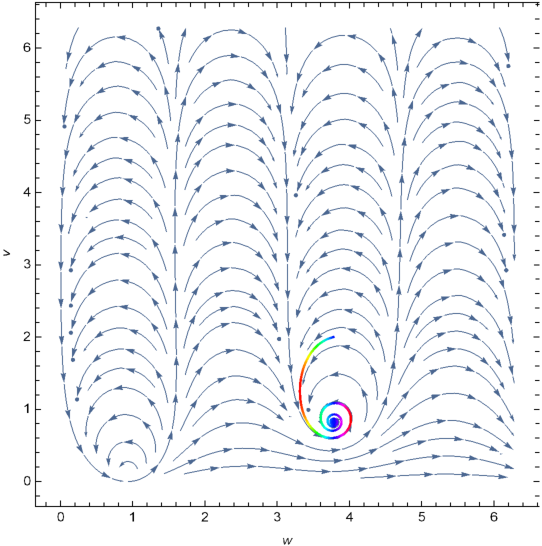}}
\subfigure[Lion Dance flow, $N=4$.]{\includegraphics[width=0.46\textwidth]{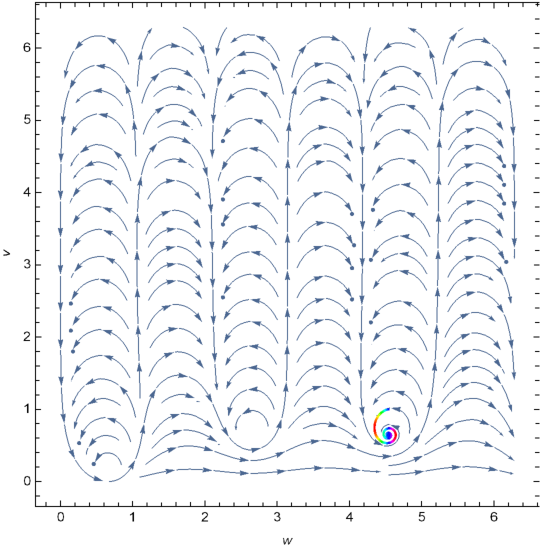}}
\subfigure[Lion Dance flow, $N=5$.]{\includegraphics[width=0.46\textwidth]{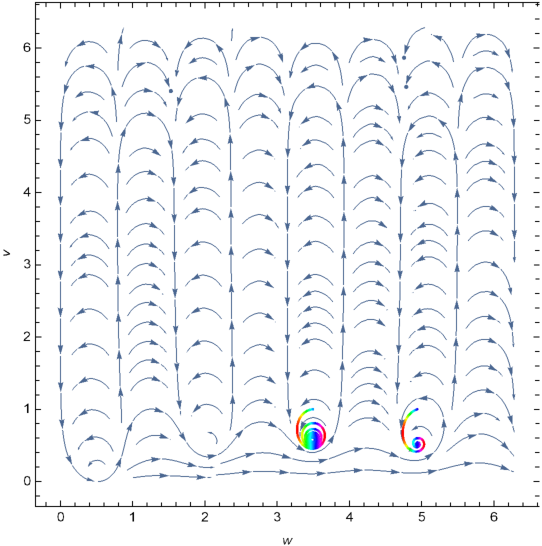}}
\subfigure[Lion Dance flow, $N=6$.]{\includegraphics[width=0.46\textwidth]{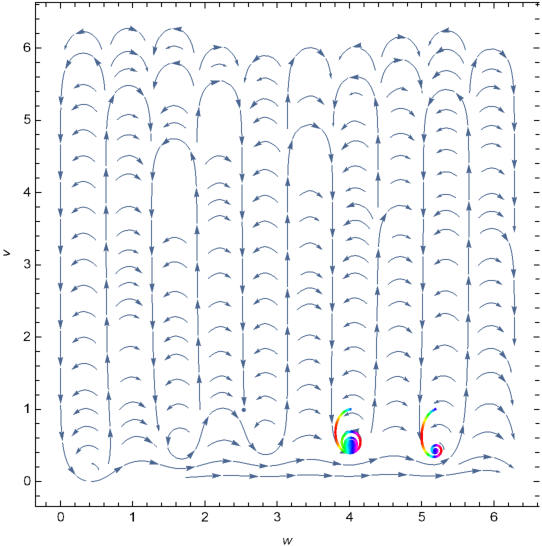}}
\caption{The vector fields with some sample trajectories of Lion Dance flow for $N=3$ (a), $N=4$ (b), $N=5$ (c), $N=6$ (d) when the coupling strength is ``critically weak." We verified in this paper that for $N=3$, the number of asymptotically stable phase-locking states of Lion Dance flow is $1$ when the coupling strength is ``super weak" or ``critically weak." Moreover, we conjecture that the number of asymptotically stable states is exactly $\lfloor \frac{N-1}{2}\rfloor$ for $N=4,5,...$ in the same coupling strength regime.}
\label{fig:100}
\end{figure}

\section{Acknowledgement}
We greatly acknowledge discussions with Zih-Hao Huang and Yen-Ting Lin.

\section{Data Availability Statement}
Data will be made available on reasonable request.

\section{Declaration}
We do not have any conflict of interest.
\clearpage

\bibliography{sn-bibliography}
\end{document}